\documentclass[12pt,reqno]{amsart}
\usepackage{a4wide}
\usepackage[T1]{fontenc}
\usepackage{amssymb,amsmath,amsthm,mathtools}
\usepackage{mathrsfs} 
\usepackage{enumitem}
\usepackage[dvipsnames]{xcolor}
\usepackage{graphicx}
\usepackage{hyperref}
\usepackage{cleveref} 
\hypersetup{colorlinks=true, linkcolor=blue, citecolor=teal, urlcolor=cyan}
\usepackage{comment}
\usepackage{amsrefs} 

\newtheorem{theorem}{Theorem}[section]
\newtheorem{lemma}[theorem]{Lemma}
\newtheorem{corollary}[theorem]{Corollary}
\newtheorem{proposition}[theorem]{Proposition}

\newtheorem{fact}[theorem]{Fact}
\newtheorem{claim}[theorem]{Claim}
\newtheorem*{claim*}{Claim}
\newtheorem{problem}[theorem]{Problem}
\newcounter{maintheorem}

\newtheorem{mainth}[maintheorem]{Theorem}
\theoremstyle{remark}
\newtheorem{remark}[theorem]{Remark}
\theoremstyle{definition}
\newtheorem{definition}[theorem]{Definition}

\numberwithin{equation}{section}
\makeatother

\newcommand{\R}{\mathbb{R}}
\newcommand{\ZZ}{\mathbb{Z}}
\newcommand{\N}{\mathbb{N}}
\newcommand{\e}{\varepsilon}
\newcommand{\p}{\varphi}
\newcommand{\n}{\left\Vert\cdot\right\Vert}
\newcommand{\nn}[1]{{\left\vert\kern-0.25ex\left\vert\kern-0.25ex\left\vert #1 
\right\vert\kern-0.25ex\right\vert\kern-0.25ex\right\vert}}

\renewcommand{\leq}{\leqslant}
\renewcommand{\geq}{\geqslant}
\renewcommand{\tilde}{\widetilde}
\newcommand{\spn}{{\rm span}}
\DeclareMathOperator{\supp}{supp}
\DeclareMathOperator{\dens}{dens}
\DeclareMathOperator{\dist}{dist}

\newcommand{\X}{\mathcal{X}}
\newcommand{\Y}{\mathcal{Y}}
\newcommand{\Z}{\mathcal{Z}}
\newcommand{\D}{\mathcal{D}}
\newcommand{\G}{\mathcal{G}}
\renewcommand{\H}{\mathcal{H}}
\newcommand{\T}{\mathcal{T}}
\newcommand{\B}{\mathcal{B}}
\newcommand{\V}{\mathcal{V}}
\newcommand{\W}{\mathcal{W}}
\renewcommand\qedsymbol{$\blacksquare$} 

\title[Lattice tilings of Hilbert spaces]{Lattice tilings of Hilbert spaces}

\author[C.A.~De~Bernardi]{Carlo Alberto De Bernardi}
\address[C.A.~De~Bernardi]{Dipartimento di Matematica per le Scienze economiche, finanziarie ed attuariali, Universit\`a Cattolica del Sacro Cuore, 20123 Milano, Italy \newline
\href{https://orcid.org/0000-0002-9654-1324}{\texttt{ORCID:0000-0002-9654-1324}}}
\email{carloalberto.debernardi@unicatt.it, carloalberto.debernardi@gmail.com}

\author[T.~Russo]{Tommaso Russo}
\address[T.~Russo]{Universit\"{a}t Innsbruck, Department of Mathematics, Technikerstra\ss e 13, 6020 Innsbruck, Austria \newline
\href{https://orcid.org/0000-0003-3940-2771}{\texttt{ORCID:0000-0003-3940-2771}}}
\email{tommaso.russo@uibk.ac.at, tommaso.russo.math@gmail.com}

\author[J.~Somaglia]{Jacopo Somaglia}
\address[J.~Somaglia]{Politecnico di Milano, Dipartimento di Matematica, Piazza Leonardo da Vinci 32, 20133 Milano, Italy \newline
\href{https://orcid.org/0000-0003-0320-3025}{\texttt{ORCID:0000-0003-0320-3025}}}
\email{jacopo.somaglia@polimi.it}

\thanks{The research of the authors has been partially supported by the GNAMPA (INdAM -- Istituto Nazionale di Alta Matematica).}

\begin{document}
\begin{abstract} We construct a bounded and symmetric convex body in $\ell_2(\Gamma)$ (for certain cardinals $\Gamma$) whose translates yield a tiling of $\ell_2(\Gamma)$. This answers a question due to Fonf and Lindenstrauss. As a consequence, we obtain the first example of an infinite-dimensional reflexive Banach space that admits a tiling with balls (of radius $1$). Further, our tiling has the property of being point-countable and lattice (in the sense that the set of translates forms a group). The same construction performed in $\ell_1(\Gamma)$ yields a point-$2$-finite lattice tiling by balls of radius $1$ for $\ell_1(\Gamma)$, which compares to a celebrated construction due to Klee. We also prove that lattice tilings by balls are never disjoint and, more generally, each tile intersects as many tiles as the cardinality of the tiling. Finally, we prove some results concerning discrete subgroups of normed spaces. By a simplification of the proof of our main result, we prove that every infinite-dimensional normed space contains a subgroup that is $1$-separated and $(1+\e)$-dense, for every $\e>0$; further, the subgroup admits a set of generators of norm at most $2+\e$. This solves a problem due to Swanepoel and yields a simpler proof of a result of Dilworth, Odell, Schlumprecht, and Zs\'ak. We also give an alternative elementary proof of Stepr\={a}ns' result that discrete subgroups of normed spaces are free.
\end{abstract}
\maketitle

\section{Introduction}
A tiling of a normed space $\X$ is a family of bodies that cover $\X$ and have mutually disjoint non-empty interiors. (We refer to \Cref{sec: prelim} for a more detailed account on the various notions that we don't define in this section.) The study of tilings in finite dimensional normed spaces, most notably the Euclidean plane, can be traced back to the beginnings of Geometry and even made its way into the visual arts, notably in the work of Escher, \cites{Behrends, Escher}. Several elementary introductions to the topic are available in the literature and we refer, \emph{e.g.}, to \cites{Adams_tiling, Behrends, GruShe} for an overview of the area. Remarkably, tilings in finite dimensions tend to be periodic and the construction of aperiodic tilings has been a long-standing open problem, especially due to its connection to the modelling of quasicristals, \cites{AxelGratias, Kellendonk, Radin, Sadun, Trevino}. Nowadays, several constructions of aperiodic tilings are available, by rotations and translations of a large finite set of tiles, \cites{Berger, Robinson, Wang}, by rotations and translations of two tiles (Penrose' `kites' and `darts'), \cites{D'Andrea, Gardner, Penrose1, Penrose2}, by translations and rotations of a single tile, \cites{monotile1, monotile2}, and, finally, only by translations of a single tile, \cites{GreenfeldTao}.

The situation is radically different in infinite dimensions and the only example of an infinite-dimensional normed space that comes naturally equipped with a well-behaved tiling is $c_0$, where it is enough to translate the unit ball by the even integers grid. Because of this reason, several attempts have been made to produce tilings that are as regular as possible, also in infinite-dimensions. Perhaps the most successful such an attempt is Preiss' result \cite{Preiss} that $\ell_2$ admits a normal tiling (in the sense that the inner radii and the diameters of the convex bodies are uniformly bounded). An earlier result, valid for every normed space, but with only radii uniformly bounded from below, was obtained by Fonf, Pezzotta, and Zanco in \cite{FPZ_BLMS}; instead, a tiling with convex bodies having uniformly bounded diameters is possible in spaces having the Radon-Nikod\'ym property, \cite{FonfLind}. A variation of the argument in \cite{FPZ_BLMS} permits to obtain a point-finite tiling with bounded convex bodies (but with no lower bound on the inner radii), \cite{MarZan}. Our main result, \Cref{mainth: ell2 body} below, gives an example of a periodic tiling in some (non-separable) Hilbert spaces (as it is more customary in the literature of infinite-dimensional tilings, we will write \emph{lattice} instead of \emph{periodic}).
\smallskip

On the other hand, there are several results asserting the impossibility to tile with bodies having good geometric properties. For instance, an elegant deduction from Sierpi\'nski's theorem on continua \cite{Sierpinski} is that a separable normed space cannot admit a tiling by rotund bodies, \cite{KleeMalZan}. Dually, Klee and Tricot proved that separable Banach spaces do not admit tilings by smooth and bounded convex bodies, \cite{KleeTri}. Outside of the separable realm, another result from \cite{KleeMalZan} is that neither uniformly convex nor uniformly smooth Banach spaces admit tilings by balls whose radii are bounded below. Recently, this result has been simplified and largely generalised to the fact that Fr\'echet smooth or LUR Banach spaces do not admit tilings by balls, \cite{DEVEtiling}. In the same paper it is also shown that normed spaces whose unit ball has some LUR point don't admit tilings by balls with radii bounded below.

There also are results where `local' properties prevent the existence of a tiling, or, more generally, of a covering. The seminal result in this direction is Corson's theorem \cite{Corson} that infinite-dimensional reflexive Banach spaces do not admit locally finite coverings by bounded convex bodies. This has been generalised in \cite{FZ_PAMS06} to the assertion that only $c_0$-saturated Banach spaces might admit such coverings. Further, Fonf \cite{Fonf3poly} proved that a separable Banach space admits a locally finite \emph{tiling} by bounded convex bodies if and only if it is isomorphically polyhedral. Other properties, such as point-finiteness, or star-finiteness have been investigated and will play a role in our paper later. Answering a question of Klee \cite{Klee1}, it was proved in \cite{FZHilbert} that the separable Hilbert space does not admit a point-finite covering by balls; this was later extended to $L^p$ spaces in \cite{FonfLevZan14}. The first-named author gave in \cite{DEBEpointfinite} a simplified proof of these negative results, showing in particular that $\ell_p(\Gamma)$ ($1\leq p<\infty$) does not admit a point-finite covering by balls whenever $|\Gamma|<\mathfrak{c}$. We refer, \emph{e.g.}, to \cites{DESOVEstar, DEVEtiling, DEJSnormal2, Klee2, Zanco_carpet} for more information on coverings and tilings in infinite dimensions.
\smallskip

For our perspective, the most interesting result is the celebrated construction, due to Klee \cites{Klee1, Klee2}, of a disjoint tiling of $\ell_1(\Gamma)$ with balls of radius $1$, for all cardinals $\Gamma$ such that $\Gamma^\omega= \Gamma$. The additional, and strikingly counter-intuitive, property that the tiling is disjoint implies that the centers of the balls constituting the tiling form a discrete Chebyshev set (namely, for every $x\in \ell_1(\Gamma)$ there exists one and only one point at minimal distance). This seminal result has influenced several other results on tilings, such as \cites{DEVEtiling, KleeMalZan}, and Klee's method of proof also plays an important role in our arguments.

To the best of our knowledge, the only known examples of infinite-dimensional normed spaces that admit tilings with balls (not necessarily of the same radius) are those based on the $c_0$ ball or on Klee's tiling of $\ell_1(\Gamma)$. In fact, translating the unit ball by the even integer grid also gives a tiling of $\ell_\infty(\Gamma)$ and of some suitable subspaces, such as $c_0(\Gamma)$, or the subspace $\ell_\infty^c(\Gamma)$ of countably supported bounded functions. Similarly, Klee's construction also applies to normed spaces of the form $\X\oplus_1\ell_1(\Gamma)$, whenever $\Gamma^\omega= \Gamma$ and $\dens(\X)\leq \Gamma$ (we will say more on this in \Cref{thm: tiling if ell1 inside}). Since all these spaces are very far from being reflexive, the following problem arises naturally.

\begin{problem}[Fonf and Lindenstrauss, \cites{FonfLind, GMZ}]\label{probl: FL} Does there exist an infinite-dimensional reflexive Banach space that can be tiled by translates of a bounded convex body?
\end{problem}

The present formulation is taken from the monograph \cite{GMZ}, where it is attributed to Fonf and Lindenstrauss; the formulation in \cite{FonfLind} actually asks the extra condition that the space is separable. Yet another formulation of the problem has been recorded in \cites{Zanco_carpet} and some remarks on possible approaches to a solution can also be found in the Handbook's article \cite{FLP_Handbook}. The main result of our paper is that such a tiling is indeed possible, even in a (non-separable) Hilbert space. More precisely, we have the following result.

\begin{mainth}\label{mainth: ell2 body} For every cardinal $\Gamma$ such that $\Gamma^\omega= \Gamma$, $\ell_2(\Gamma)$ admits a lattice tiling by translates of a symmetric and bounded convex body.

Therefore, there exists an equivalent norm $\nn\cdot$ on $\ell_2(\Gamma)$ such that $(\ell_2(\Gamma), \nn\cdot)$ admits a lattice tiling by balls.
\end{mainth}

Plainly, every symmetric and bounded convex body is the unit ball of an equivalent norm. Therefore, the second clause follows immediately from the first one. As it is apparent, not only our main result answers \Cref{probl: FL}, but we also get a stronger result. First of all, our construction yields the first example of an infinite-dimensional reflexive (even isomorphic to a Hilbert space) Banach space that admits a tiling with balls. Further, we also have the additional regularity property that the tiling is lattice (see \Cref{def: lattice}). This property is obviously shared with the tiling of $c_0$, but, as we will explain later, is not present in Klee's tiling. Therefore, in this sense, our tiling is more regular than the original construction in Klee and resembles the periodic tilings in finite dimensions from the first paragraph.

One additional interesting property is that our tiling is point-countable (\Cref{prop: point-countable}); further, our construction leads to a point-countable covering of $\ell_2(\Gamma)$ with balls of radius $1$ (in the canonical norm). The relevance of this result is that, when $\Gamma< \mathfrak{c}$, $\ell_2(\Gamma)$ does not admit point-finite coverings by balls, \cites{DEBEpointfinite, FZHilbert}, and it is an important open problem whether $\ell_2(\mathfrak{c})$ admits such a covering. While we don't solve this problem, we have the weaker result that a point-countable covering by balls of radius $1$ is possible.
\smallskip

Our construction, whose main idea we explain below, admits a counterpart for $\ell_1(\Gamma)$, where we obtain the following result.

\begin{mainth}\label{mainth: ell1 balls} For every cardinal $\Gamma$ such that $\Gamma^\omega= \Gamma$, $\ell_1(\Gamma)$ admits a point-$2$-finite lattice tiling by balls (in its original norm).
\end{mainth}

This tiling resembles very closely the one by Klee, albeit with two differences. On the one hand, Klee's tiling is disjoint, while on the other one ours is a lattice tiling. At this point it is natural to ask whether the construction by Klee did implicitly produce a lattice tiling, or whether our construction could be improved to give a disjoint tiling. As it turns out, it is an easy observation (\Cref{prop: lattice not disjoint}), that lattice tilings by balls are never disjoint. Hence, our construction has an extra regularity feature that is not present in Klee's one. Further, even though our tiling cannot be disjoint, it is as close to being disjoint as possible because each point belongs to at most $2$ tiles and the only points that could belong to more than one tile are the extreme points of the tiles (\Cref{thm: only vertices touch}).

While the fact that lattice tilings are not disjoint is an easy observation, it admits far-reaching generalisations when one ponders how many tiles must each tile intersect (which is a natural question, in the light of the results in \cites{Breen, DESOVEstar, Nielsen}). We study this question in \Cref{sec: intersections}, where we obtain in particular that if $\T$ is a lattice tiling by balls of an infinite-dimensional normed space, then each tile intersects $|\T|$-many other tiles. This nicely constrasts with the fact that our tilings of $\ell_2(\Gamma)$ and $\ell_1(\Gamma)$ are point-countable and point-$2$-finite respectively: while pointwise the number of intersections is controlled, on each tile the amount of intersections is as large as possible.
\smallskip

We now briefly explain the main strategy behind the proof of \Cref{mainth: ell2 body} and \Cref{mainth: ell1 balls}. Klee's proof in \cite{Klee1} essentially consists in building a subset $\D$ of $\ell_p(\Gamma)$ ($1\leq p<\infty$) that is $(2^{1/p}+)$-separated and $1$-dense. From this point, a standard approach is to consider the Voronoi cells associated to $\D$, which directly give the desired tiling. The approach through Voronoi cells is also present in several other papers, such as \cites{DEGBnormal1, DEJSnormal2, Preiss}. Our strategy consists in improving Klee's construction by building a $2^{1/p}$-separated and $1$-dense \emph{subgroup} of $\ell_p(\Gamma)$ (which, when $p>1$, is also $(2^{1/p}+)$-separated, but not when $p=1$). After having such a subgroup at our disposal, we consider the associated Voronoi cells and the group structure easily yields the lattice property (\Cref{prop: tile by Voronoi}). Therefore, the main ingredient in our argument, and third main result of the paper, is the following.

\begin{mainth}\label{mainth: subgroups}${}$
\begin{enumerate}
    \item\label{mth: subgroup ell_p} Let $\Gamma$  be a cardinal number with $\Gamma^\omega= \Gamma$ and $1\leq p< \infty$. Then, $\ell_p(\Gamma)$ contains a $2^{1/p}$-separated and $1$-dense subgroup. Furthermore, if $p>1$, the subgroup can be chosen to be $(2^{1/p}+)$-separated.
    \item\label{mth: subgroup general} Let $\X$ be an infinite-dimensional normed space. Then for every $\e>0$ there exists a subgroup $\D$ of $\X$ that is $1$-separated and $(1+\e)$-dense. Furthermore, the group is generated by a set of vectors of norm at most $2+\e$.
\end{enumerate}
\end{mainth}

Item \eqref{mth: subgroup ell_p} is the above mentioned tool for our tilings, while \eqref{mth: subgroup general} is obtained from a simplification of the argument for \eqref{mth: subgroup ell_p} and connects to several papers concerning discrete subgroups of normed spaces. To begin with, there is an extensive literature concerning discrete subgroups of $\R^n$ that are $r$-dense for some $r>0$. This is usually phrased in different terms, via the notion of simultaneous packing and covering; we shall not enter this and just refer, \emph{e.g.}, to \cites{Rogers_book, Zong1, Zong2} and references therein. In the infinite-dimensional framework, Rogers \cite{Rog84}*{Theorem 2} proved that every Banach space contains, for every $\e>0$, a subgroup which is 1-separated and $(\frac{3}{2}+\e)$-dense; further, it admits a set of generators of norm at most $2+\e$. It was asked in \cite{Swane09}*{Section 6.1} (using the packing terminology) whether the constant $(\frac{3}{2}+\e)$ could be replaced with $(1+\e)$. For separable Banach spaces, this was answered by Dilworth, Odell, Schlumprecht, and Zs\'ak in \cite{DOSZ}*{Theorem 5.5}: every separable Banach space contains, for every $\e>0$, a subset of the unit ball, whose generated group is $1/3$-dense and $(3+\e)^{-1}$-separated. The argument in \cite{DOSZ} depends on the existence of $(1+\e)$-bounded M-bases, hence it cannot be extended to the non-separable context.

Therefore, \Cref{mainth: subgroups}\eqref{mth: subgroup general} improves Roger's result by answering \cites{Swane09} for all normed spaces. Further, we have a different proof of the result in \cites{DOSZ}, that also has the advantage of applying to all normed spaces and of being much shorter. Incidentally, if we rescale the generators of our group to the unit ball, we obtain a $1/2$-dense and $(2+\e)^{-1}$-separated subgroup; hence, we also improve their factor $1/3$.
\smallskip

In conclusion to this section, let us briefly explain the structure of our paper. \Cref{sec: prelim} revises the relevant definitions from the theory of tilings and Voronoi cells. \Cref{sec: main tile} contains our main results: we prove \Cref{mainth: subgroups}\eqref{mth: subgroup ell_p}, we deduce \Cref{mainth: ell2 body} and \Cref{mainth: ell1 balls}, and study more properties of such tilings. In \Cref{sec: intersections} we show that each tile in a lattice tiling by balls intersects `many' other tiles. \Cref{sec: subgroup} is dedicated to discrete subgroups of normed spaces; we prove \Cref{mainth: subgroups}\eqref{mth: subgroup general} and we also give an alternative self-contained proof of Stepr\={a}ns' result \cite{Steprans} that discrete subgroups of normed spaces are free. Finally, we collect or reiterate some problems that arise from our research in \Cref{sec: problems}.

\section{Preliminaries}\label{sec: prelim}

Throughout the paper we only consider infinite-dimensional normed spaces over the real field, even when not explicitly assumed. We denote by $B_\X$ the closed unit ball of a normed space $\X$; in case an equivalent norm $\nn\cdot$ is introduced on $\X$, the unit ball in the new norm is indicated by $B_{\nn\cdot}$ and similarly for the unit sphere $S_{\nn\cdot}$. We also write $B_r(x)$ or $x+ rB_\X$ for the closed ball centred at $x$ and with radius $r$. For $x,y\in \X$, $[x,y]$ denotes the closed segment in $\X$ with endpoints $x$ and $y$, and $(x,y)=[x,y]\setminus\{x,y\}$.

The cardinality of a set $S$ is indicated by $|S|$. We adopt the convention to regard cardinal numbers as initial ordinal numbers; hence, we write $\omega$ for the cardinal $\aleph_0$. We also denote by $\mathfrak{c}$ the cardinality of continuum. By $\N$ we denote the set of strictly positive integers. Given a cardinal $\Gamma$, by $\Gamma^\omega$ we understand cardinal exponentiation, namely the cardinality of all sequences with values in $\Gamma$. Let us recall the following classical theorem of Schmidt \cite{Schmidt} and Stone \cite{Stone} (see, \emph{e.g.}, \cite{Hodel}*{Section 8}): if every non-empty open set in a complete metric space $\mathcal{M}$ has weight $\Gamma$, then $|\mathcal{M}|= \Gamma^\omega$. In particular, our ubiquitous assumption that $\Gamma^\omega= \Gamma$ is made in order to assure that $|\ell_p(\Gamma)|= \Gamma$, for every $1\leq p< \infty$. Notice that under the generalised continuum hypothesis, $\Gamma^\omega= \Gamma$ is satisfied if and only if $\Gamma$ has uncountable cofinality, \cite{Jech}*{Theorem 5.15}. We also point out that a similar statement can be found in \cite{DEVEtiling}*{Remark~5.5}, but stated improperly. 
\smallskip

We now recall the notions concerning tilings that we consider in our article. A \emph{body} in a normed space is a set that is the closure of its non-empty interior. In particular, a \emph{convex body} is a closed convex set with non-empty interior. A \emph{tiling} of a normed space $\X$ is a family of bodies that cover $\X$ and have mutually disjoint interiors (sometimes it is said that the bodies are \emph{non-overlapping}). Our only interest in most of the paper (with the only exception of \Cref{sec: p not 2}) is on tilings by symmetric and bounded convex bodies, \emph{i.e.}, balls of equivalent norms. Furthermore, our tilings will be obtained by translating a single tile under the action of a subgroup.

\begin{definition}\label{def: lattice} A tiling $\T$ of a normed space $\X$ is a \emph{lattice tiling} if there are a body $B$ and a (discrete) subgroup $\D$ of $\X$ such that $\T= \{d+ B\colon d\in \D \}$. In the case when $B$ is the unit ball $B_\X$, we say that $\T$ is a \emph{lattice tiling by balls}.
\end{definition}

Notice that if a lattice tiling $\T$ has the form $\{d+ B\colon d\in \D \}$, where $B=B_r(x)$ is a ball, then we can translate and rescale $B_r(x)$ to $B_\X$ and consider instead the tiling $\{d/r+ B_\X\colon d\in \D \}$. Therefore, our definition of lattice tiling by balls is essentially equivalent to requiring that the body $B$ is a ball.

Besides this regularity property on the disposition of the tiles, we will consider two more types of properties of a tiling: one obtained by requiring stronger geometric assumption on the tiles and one requiring some control on the intersections. For example, we will consider tilings with or rotund, or smooth, convex bodies. Concerning intersection properties of a tiling, we will make use of the following notions, that we recall for a general family.

\begin{definition}\label{D: finiteness} A family $\mathcal F$ of subsets of a normed space $\X$ is said to be:
\begin{enumerate}
\item {\em star-finite} if each of its members intersects only finitely many members of $\mathcal F$;
\item {\em point-finite} (respectively, {\em point-countable}) if each $x\in \X$ belongs to at most finitely many (respectively, countably many) members of $\mathcal F$;
\item {\em locally finite} if each $x\in \X$ admits a neighbourhood that intersects finitely many elements of $\mathcal{F}$.
\end{enumerate}
\end{definition}

The notion of point-finiteness might also be quantified. For $n\in \N$, we say that a family $\mathcal F$ is \emph{point-$n$-finite} if every point of $\X$ belongs to at most $n$ elements of the family $\mathcal F$. This notion is called $n$-finite in \cite{DESOVEstar}, but we prefer the name point-$n$-finite since one might also consider star-$n$-finite, or locally $n$-finite families.

We will also need the notion of singular point for a tiling (or more generally, for a family of subsets). A point $x\in \X$ is a \emph{singular point} for a family $\mathcal{F}$ if every neighbourhood of $x$ intersects infinitely many elements of $\mathcal{F}$; a point is \emph{regular} when it is not singular. Clearly, the family $\mathcal{F}$ is locally finite if and only if it doesn't admit any singular point. Similarly, for a cardinal $\Gamma$ a \emph{$\Gamma$-singular point} is a point every whose neighbourhood intersects $\Gamma$-many elements of $\mathcal{F}$.
\smallskip

Let $\D$ be a subset of a normed space $\X$ and $r>0$. The set $\D$ is \emph{$r$-separated} (respectively, \emph{$(r+)$-separated}) if $\|d- h\|\geq r$ (respectively, $\|d- h\|> r$) for any distinct elements $d,h\in \D$. $\D$ is \emph{$r$-dense} if for every $x\in \X$ there is $d\in \D$ such that $\|x-d\|\leq r$. The set $\D$ is \emph{proximinal} if for every $x\in \X$ there exists $d\in \D$ such that $\|x-d\|= \dist(x,\D)$. For each $d\in \D$, we define the \emph{Voronoi cell} $V_d$ by
\[ V_d\coloneqq \{x\in \X\colon \|x-d\|\leq \|x-h\|, \mbox{ for all } h\in \D\}. \]
When $\D$ is a subgroup, being $r$-separated, for some $r>0$, is plainly equivalent to being discrete and we shall use both terminologies in the paper.

As we mentioned in the Introduction, a well-known method to produce tilings of normed spaces is to construct a suitable set $\D$ and consider the associated Voronoi cells. In the following folklore proposition we summarise the properties of $\D$ needed for the Voronoi cells $\{V_d\}_{d\in \D}$ to be a tiling and some properties of this tiling. Most clauses in the proposition can be found in \cite{Klee2}*{Theorem 3.1}, but we sketch part of the proof for the sake of completeness.

\begin{proposition}\label{prop: tile by Voronoi} Let $\X$ be a normed space and $\D\subseteq \X$ such that:
\begin{enumerate}
    \item\label{i: prox} $\D$ is proximinal;
    \item\label{i: subgroup} $\D$ is a subgroup of $\X$;
    \item\label{i: separated} $\D$ is $R$-separated and $r$-dense.
\end{enumerate}
Then the associated Voronoi cells  $\{V_d\}_{d\in \D}$ satisfy the following:
\begin{enumerate}[label=\normalfont{(\arabic*)},ref=\arabic*]
    \item\label{i: renorming} $V_0$ is a closed, symmetric, starshaped set such that
    \begin{equation}\label{eq: equiv norm}
        \frac{R}{2} B_\X\subseteq V_0 \subseteq r B_\X.
    \end{equation}
    \item\label{i: invariant} For each $d\in \D$, $V_d=d+ V_0$.
    \item\label{i: covering} $\{V_d\}_{d\in \D}$ is a covering of $\X$.
\end{enumerate}
Further, if $\X$ is LUR, then each $V_d$ is the closure of its interior (hence, a body) and $\{V_d\}_{d\in \D}$ is a tiling. Finally, if $\X$ is a Hilbert space, each $V_d$ is convex.
\end{proposition}

\begin{proof} The fact that $\{V_d\}_{d\in \D}$ is a covering follows immediately from \eqref{i: prox}. In fact, for a fixed $x\in \X$, \eqref{i: prox} yields the existence of $d\in \D$ such that $\|x-d\|\leq \|x-h\|$, for all $h\in \D$. Hence, $x\in V_d$. That each $V_d$ is closed and starshaped with respect to $d$ is easy to check. Likewise, if $\X$ is a Hilbert space, $V_d$ is the intersection of the half-spaces
\[ \left\{x\in \X\colon \langle x, h-d\rangle\leq \tfrac{1}{2}\big(\|h\|^2 - \|d\|^2\big)\right\} \]
for $h\in \D$ and it is therefore convex. The clause concerning LUR spaces is less immediate and we refer to \cite{Klee2}*{Theorem 3.1}. We now show how \eqref{i: renorming} and \eqref{i: invariant} follow from \eqref{i: subgroup} and \eqref{i: separated}.

If $\D$ is a subgroup and $d\in \D$, we have $\D-d= \D$, hence
\begin{align*}
    V_d \coloneqq& \{x\in \X\colon \|x-d\|\leq \|x-h\|, \mbox{ for all } h\in \D\}\\
    =& \{x+d\in \X\colon \|x\|\leq \|x-(h-d)\|, \mbox{ for all } h\in \D\}\\
    =& \{x+d\in \X\colon \|x\|\leq \|x-h\|, \mbox{ for all } h\in \D\}= d+V_0.
\end{align*}
Similarly, $\D= -\D$ implies that $V_0$ is symmetric.

Finally, we use \eqref{i: separated} to show \eqref{eq: equiv norm}. Take $x\in \X$ with $\|x\|\leq R/2$. Since $\D$ is $R$-separated and $0\in \D$, $\|d\|\geq R$ for every non-zero $d\in \D$. Thus, $\|x-d\|\geq R/2\geq \|x\|$ for every non-zero $d\in \D$, which implies that $x\in V_0$. For the second inclusion, take $x\in \X$ with $\|x\|>r$. By definition, there exists $d\in \D$ such that $\|x-d\|\leq r$; hence, $\|x-d\|< \|x\|$, which implies that $x\notin V_0$.
\end{proof}

In our argument the set $\D$ will be constructed by transfinite induction and conditions \eqref{i: subgroup} and \eqref{i: separated} in \Cref{prop: tile by Voronoi} will be preserved at each step of the induction. On the other hand, the proximinality of $\D$ is a global condition, which is much harder to achieve by induction. In our context, we shall obtain proximinality by means of the following lemma, also due to Klee \cite{Klee1}*{Remark 2.3}, involving the Kottman constant of a normed space. Recall that the \emph{Kottman constant} $K(\X)$ \cite{Kottman} of a normed space $\X$ is
\[ K(\X)\coloneqq \sup\big\{r>0\colon B_\X \mbox{ contains an } r\mbox{-separated sequence} \big\}. \]
An elementary computation, essentially based on a sliding bump argument, shows that $K(\ell_p(\Gamma))= 2^{1/p}$, for every $1\leq p<\infty$ and every infinite set $\Gamma$, \cites{BRR, Kottman_TAMS}. Further, by James' distortion  theorem, $K(\X)=2$ for every Banach space that contains an isomorphic copy of $c_0$, or $\ell_1$. For recent results on Kottman's constant we refer to \cites{CGKP, CGP, HKR_JFA, R_RACSAM} and references therein.

\begin{lemma}[Klee, \cite{Klee1}]\label{lem: K(X) proximinal} Every $K(\X)$-separated and $1$-dense subset of a normed space $\X$ is proximinal.
\end{lemma}
\begin{proof} Let $\D$ be a $K(\X)$-separated and $1$-dense subset of $\X$ and take any $x\in \X$. By definition, there exists some $d\in \D$ such that $\|x-d\|\leq 1$. If there is no $h\in \D$ such that $\|x-h\|<1$, then clearly the element $d$ has minimal distance from $x$. Therefore, we can assume without loss of generality that there is $d\in \D$ such that $r\coloneqq \|x-d\|<1$. Thus, $B_r(x)\cap \D$ is non-empty. Since $r<1$, the set $B_r(x)$ cannot contain an infinite $K(\X)$-separated set, because otherwise $B_\X$ would contain an infinite $K(\X)/r$-separated set, contrary to the definition of $K(\X)$. Therefore, $B_r(x)\cap \D$ is a finite set and, hence, there is a point in $B_r(x)\cap \D$ that has minimal distance from $x$.
\end{proof}

\section{Discrete subgroups and lattice tilings in \texorpdfstring{$\ell_p(\Gamma)$}{ellp(Gamma)}}\label{sec: main tile}
The purpose of this section is to prove our main results concerning tilings (\Cref{mainth: ell2 body} and \Cref{mainth: ell1 balls}). As we explained in the Introduction, the main part of the argument will consist in the construction of a certain subgroup of $\ell_p(\Gamma)$, that we perform in the first part of the section. Therefore, this section also contains the proof of \Cref{mainth: subgroups}\eqref{mth: subgroup ell_p}. In the second part of the section we combine this result with the facts concerning Voronoi cells that we recalled in \Cref{sec: prelim} and we obtain the desired tilings of $\ell_2(\Gamma)$ and $\ell_1(\Gamma)$. We also study some additional properties of such tilings.

\begin{theorem}\label{thm: ell_p subgroup} Let $\Gamma$  be a cardinal number with $\Gamma^\omega= \Gamma$ and $1\leq p< \infty$. Then, $\ell_p(\Gamma)$ contains a $2^{1/p}$-separated and $1$-dense subgroup.

Furthermore, if $p>1$, the subgroup can be chosen to be $(2^{1/p}+)$-separated.
\end{theorem}

We will see below (\Cref{rmk: p=1 not strict} and \Cref{prop: lattice not disjoint}) that the assumption that $p>1$ is essential in the last part of the theorem.

\begin{proof} Fix a cardinal number $\Gamma$ such that $\Gamma^\omega= \Gamma$ and $1\leq p<\infty$. The result of Schmidt and Stone that we recalled in \Cref{sec: prelim} ensures us that the cardinality of $\ell_p(\Gamma)$ equals $\Gamma$. Therefore, we can find an injective enumeration $(u_\alpha)_{\alpha< \Gamma}$ with $u_0=0$ and such that $\ell_p(\Gamma)= \{u_\alpha\}_{\alpha< \Gamma}$. We shall now build by transfinite induction an increasing chain $(\D_\alpha)_{\alpha< \Gamma}$ of subgroups of $\ell_p(\Gamma)$ with the properties that, for all $\alpha<\Gamma$,
\begin{enumerate}
    \item\label{3.1.i: septd} $\D_\alpha$ is $2$-separated for $p=1$ and $(2^{1/p}+)$-separated for $p\in (1,\infty)$;
    \item\label{3.1.ii: card} $|\D_\alpha|\leq \max \{|\alpha|,\omega\}$;
    \item\label{3.1.iii: close} there exists $d\in \D_\alpha$ with $\|u_\alpha- d\|_p \leq 1$.
\end{enumerate}
Once such a chain has been constructed, the desired subgroup is just the union of the chain, $\D\coloneqq \bigcup_{\alpha<\Gamma} \D_\alpha$. Indeed, the fact that the chain $(\D_\alpha)_{\alpha< \Gamma}$ is increasing clearly implies that $\D$ is a subgroup of $\ell_p(\Gamma)$. By the same reason, \eqref{3.1.i: septd} implies that $\D$ is $2$-separated for $p=1$ and $(2^{1/p}+)$-separated for $p\in (1,\infty)$. Furthermore, condition \eqref{3.1.iii: close} yields that $\D$ is $1$-dense in $\ell_p(\Gamma)$. In fact, for every $x\in \ell_p(\Gamma)$ there is some $\alpha<\Gamma$ with $x= u_\alpha$; thus, by \eqref{3.1.iii: close}, there is $d\in \D_\alpha\subseteq \D$ with $\|x-d\|_p \leq 1$, as desired.
\smallskip

Consequently, we only have to construct a chain $(\D_\alpha)_{\alpha< \Gamma}$ as above. We set $\D_0=\{0\}$; \eqref{3.1.i: septd} and \eqref{3.1.ii: card} are obviously satisfied and so is \eqref{3.1.iii: close}, because $u_0=0$. By transfinite induction, suppose that, for some $\gamma< \Gamma$, we have already built an increasing family $(\D_\alpha)_{\alpha< \gamma}$ of subgroups of $\ell_p(\Gamma)$ such that \eqref{3.1.i: septd}--\eqref{3.1.iii: close} hold. We now distinguish two cases. If there is $d\in \bigcup_{\alpha<\gamma} \D_\alpha$ such that $\|u_\gamma- d\|_p \leq 1$, then we just set $\D_\gamma\coloneqq \bigcup_{\alpha<\gamma} \D_\alpha$. In fact, arguing as before, we see that $\D_\gamma$ is a subgroup and \eqref{3.1.i: septd} holds; the validity of \eqref{3.1.ii: card} is also clear. Finally, \eqref{3.1.iii: close} holds by our assumption.

In other words, we may (and do) assume without loss of generality that $\|u_\gamma- d\|_p >1$ for all $d\in \bigcup_{\alpha<\gamma} \D_\alpha$. Consider now the set
\[ \Gamma_0\coloneqq \bigcup \Big\{ \supp(d)\colon d\in \bigcup_{\alpha<\gamma} \D_\alpha \Big\} \cup \supp(u_\gamma). \]
According to \eqref{3.1.ii: card}, $|\bigcup_{\alpha<\gamma} \D_\alpha|\leq \max\{|\gamma|,\omega\}$; further, the support of every vector in $\ell_p(\Gamma)$ is countable. Therefore, $|\Gamma_0|\leq \max\{|\gamma|,\omega\}< \Gamma$; thus, we can find an index $\tilde{\gamma}\in \Gamma\setminus \Gamma_0$. We are now in position to define the subgroup $\D_\gamma$:
\[ \D_\gamma\coloneqq \bigcup_{\alpha<\gamma} \D_\alpha+ (u_\gamma+ e_{\tilde{\gamma}})\ZZ. \]

Plainly, $\D_\gamma$ is a subgroup of $\ell_p(\Gamma)$ that contains each $\D_\alpha$, $\alpha<\gamma$; further, $|\D_\gamma|\leq \max\{|\gamma|,\omega\}$. Moreover, by definition, the vector $u_\gamma+ e_{\tilde{\gamma}}$ belongs to $\D_\gamma$ and its distance from $u_\gamma$ equals $1$; thus, \eqref{3.1.iii: close} holds. Hence, in order to complete the proof it is enough to verify that \eqref{3.1.i: septd} holds. As $\D_\gamma$ is a subgroup, this amounts to taking any non-zero $x\in \D_\gamma$ and proving that $\|x\|_1\geq 2$ and $\|x\|_p> 2^{1/p}$ when $p>1$. Let us then write $x= d+ n(u_\gamma+ e_{\tilde{\gamma}})$, where $d\in \bigcup_{\alpha<\gamma} \D_\alpha$ and $n\in \ZZ$; notice that
\begin{equation}\label{eq: split norm}
    \|x\|_p^p= \big\|d+ nu_\gamma\big\|_p^p+ |n|^p,
\end{equation}
because the support of $d+ nu_\gamma$ is contained in $\Gamma_0$, while $\tilde{\gamma}\notin \Gamma_0$. We now have to distinguish three different cases, depending on the value of $|n|$.
\begin{itemize}
    \item If $n= 0$, then $\|x\|_p= \|d\|_p$. By \eqref{3.1.i: septd} and the transfinite induction assumption, the set $\bigcup_{\alpha<\gamma} \D_\alpha$ is $2$-separated when $p=1$ and $(2^{1/p}+)$-separated when $p>1$. Therefore, $\|x\|_1\geq 2$ and $\|x\|_p> 2^{1/p}$ for $p>1$.
    \item Instead, if $|n|\geq 2$, we clearly have $\|x\|_p\geq |n|\geq 2$. Thus, $\|x\|_1\geq 2$ and $\|x\|_p\geq2> 2^{1/p}$ for $p>1$.
    \item Finally, if $|n|=1$ we have
    \[ \|x\|_p^p= \|\pm d- u_\gamma\|_p^p +1> 2, \]
    because $\pm d\in \bigcup_{\alpha<\gamma} \D_\alpha$ and, by our previous assumption, $\|d- u_\gamma\|_p>1$ for each $d\in \bigcup_{\alpha<\gamma} \D_\alpha$. Thus, $\|x\|_p>2^{1/p}$ for each $1\leq p <\infty$.
\end{itemize}
Therefore, in each case the set $\D_\gamma$ also satisfies \eqref{3.1.i: septd}, which concludes the proof.
\end{proof}

\begin{remark}\label{rmk: p=1 not strict} Observe that in the second case of the previous proof we do not obtain a strict inequality for the separation when $p=1$. Thus, even if we assumed by transfinite induction that each $\D_\alpha$, $\alpha<\gamma$, is $(2+)$-separated, it would not follow that $\D_\gamma$ is $(2+)$-separated. As we will see in \Cref{prop: lattice not disjoint}, this is not a drawback of our method of proof: no normed space can admit a $(2+)$-separated and $1$-dense subgroup. 
\end{remark}

\begin{remark} Suppose now that $\Gamma$ is any infinite cardinal, without assuming that $\Gamma^\omega= \Gamma$. We could then repeat the above proof, with $(u_\alpha)_{\alpha< \Gamma}$ only being a dense subset of $\ell_p(\Gamma)$. The construction of the subgroups $(\D_\alpha)_{\alpha< \Gamma}$ carries over without any change. The only difference is that $\D\coloneqq \bigcup_{\alpha<\Gamma} \D_\alpha$ is no longer $1$-dense and one can only deduce that $\dist(x,\D)\leq 1$ for all $x\in \ell_p(\Gamma)$, equivalently that $\D$ is $(1+\e)$-dense, for every $\e>0$. As in the previous remark, this result cannot be improved, because $\ell_2$ does not admit a $\sqrt{2}$-separated and $1$-dense subset. In fact, by \Cref{lem: K(X) proximinal} such a set would be proximinal, but no countable proximinal subset of $\ell_2$ can be $r$-dense, \cite{FonfLind}*{Proposition 2.1}.
\end{remark}

Let us re-state explicitly the conclusion of the previous remark.
\begin{corollary} For every infinite set $\Gamma$, $\ell_p(\Gamma)$ contains a $2^{1/p}$-separated subgroup $\D$ such that $\dist(x,\D)\leq 1$ for all $x\in \ell_p(\Gamma)$.
\end{corollary}

In \Cref{thm: group in X} we shall show a simplification of the above proof that gives the same result in every normed space $\X$, albeit with worse constants.
\smallskip

We now move on to the next part of the section and apply \Cref{thm: ell_p subgroup} to our main interest in this paper, the construction of tilings in $\ell_p(\Gamma)$.

\subsection{Case \texorpdfstring{$p=2$}{p=2}}
We begin with the case when $p=2$, that leads us to the proof of \Cref{mainth: ell2 body}. Thus, let $\Gamma$ be a cardinal such that $\Gamma^\omega= \Gamma$ and, according to \Cref{thm: ell_p subgroup}, take a $(\sqrt{2}+)$-separated and $1$-dense subgroup $\D$ of $\ell_2(\Gamma)$. Because of \Cref{lem: K(X) proximinal}, we know that $\D$ is proximinal. Hence, we are in position to apply \Cref{prop: tile by Voronoi} and deduce that the Voronoi cells $\{V_d\}_{d\in \D}$ associated to $\D$ constitute a lattice tiling by symmetric convex bodies. Further, \eqref{eq: equiv norm} implies that
\[ \frac{\sqrt{2}}{2} B_{\ell_2(\Gamma)}\subseteq V_0\subseteq B_{\ell_2(\Gamma)}. \]
Consequently, $V_0$ is additionally bounded. In other words, $V_0$ is the unit ball of an equivalent norm $\nn\cdot$ on $\ell_2(\Gamma)$ and in this norm each $V_d$ is a ball of unit radius. This proves all the claims in \Cref{mainth: ell2 body}; notice that we additionally know that the Banach--Mazur distance between $\ell_2(\Gamma)$ and $(\ell_2(\Gamma),\nn\cdot)$ is at most $\sqrt{2}$.
\smallskip

We now study some additional properties of the tiling from \Cref{mainth: ell2 body}.

\begin{proposition}\label{prop: point-countable} The tiling from \Cref{mainth: ell2 body} is point-countable. Therefore, $(\ell_2(\Gamma),\nn\cdot)$ admits a point-countable (lattice) tiling with balls of radius $1$.

Further, the collection $\{d+ B_{\ell_2(\Gamma)}\}_{d\in \D}$ constitutes a point-countable covering of $\ell_2(\Gamma)$ by balls of unit radius (in the canonical norm).
\end{proposition}

As we said in the Introduction, $\ell_2(\Gamma)$ does not admit point-finite coverings by balls when $\Gamma< \mathfrak{c}$, \cites{DEBEpointfinite, FZHilbert}. We don't know whether our construction can be improved to produce a point-finite covering of $\ell_2(\mathfrak{c})$, thus this problem remains open (\Cref{probl: point-finite}). It is also worth noting that our tiling cannot be locally countable, because of \cite{FPZsingularpoints}*{Proposition 1}.

\begin{proof} Both claims are easy consequence of the fact that, for $1\leq p<\infty$, the unit ball of $\ell_p(\Gamma)$ does not contain an uncountable $(2^{1/p}+)$-separated subset, \cite{Kosz_TAMS}*{Lemma 52}. In fact, suppose that there exists a point $x\in \ell_2(\Gamma)$ the belongs to $d+ B_{\ell_2(\Gamma)}$ for uncountably many $d\in \D$. Then, the ball $B_1(x)$ contains an uncountable subset of $\D$; because $\D$ is $(\sqrt{2}+)$-separated, this contradicts \cite{Kosz_TAMS}*{Lemma 52}. Finally, since $V_d\subseteq d+ B_{\ell_2(\Gamma)}$, the first assertion follows directly from the one we just proved.
\end{proof}

We now pass to studying some geometric properties of the Banach space $(\ell_2(\Gamma),\nn\cdot)$. As it turns out, the norm $\nn\cdot$ closely resembles a polyhedral one, with many flat faces; norms with similar properties were constructed for instance in \cites{Moreno, Partington}. Before our main result that makes this precise, \Cref{thm: not rotund &co} below, we need to recall the notions of QP-point and flat point, \emph{cf.}\ \cites{AD, KleeTri}. Moreover, we introduce the notion of $\Delta$-QP-point. Roughly speaking, this notion says that around the point $x$ the convex body $B$ coincides with the intersection of finitely many half-spaces (see also \cite{DHRfund}*{Lemma~2.6}). For more well-known unexplained notions we refer, \emph{e.g.}, to \cites{DEBEPRESOMMLUR, DEBESOMALUR, GMZ2}. 

\begin{definition} Let $B$ be a convex body in a normed space $\X$, and $x\in\partial B$. Then:
\begin{enumerate}
\item  $x$ is a \emph{quasi-polyhedral} point (\emph{QP-point}) of $B$ if there exists a neighbourhood $V$ of $x$ such that $[x,y]\subseteq \partial B$ whenever $y\in V\cap\partial B$;
\item  $x$ is a \emph{$\Delta$-QP-point} of $B$ if there exist a neighbourhood $V$ of $x$ and functionals $f_1,\ldots,f_n\in \X^*$ such that 
\begin{equation} \label{eq: LFC}
    B\cap V=\bigcap_{k=1}^n \{z\in V\colon f_k(z)\leq f_k(x)\};
\end{equation}
\item $x$ is a \emph{flat} point of $B$ if there exists a supporting hyperplane $H$ to $B$ at $x$ such that $x\in\text{a-int}_H(B\cap H)$, where $\text{a-int}_H(B\cap H)$ is the relative algebraic interior (or `core') of $B\cap H$ in $H$. If $\X$ is a Banach space then, by the Baire theorem, it is equivalent to require that $x$ belongs to $\text{int}_H(B\cap H)$. 
\end{enumerate}
\end{definition}

The name $\Delta$-QP-point is justified by the fact that, when $B$ is the unit ball of $\X$, $x$ is a $\Delta$-QP-point for $B$ if and only if it is a QP-point and property ($\Delta$) is satisfied at $x$. Since we won't use this characterisation, we omit its proof and we refer to \cite{FOVEpoly} for the definition of ($\Delta$). In the following lemma, whose proof is immediate, we collect some basic facts concerning the notions we just introduced.

\begin{lemma}\label{obs: flat VS QP} Let $B$ be a convex body in a normed space $\X$. Then:
\begin{enumerate}
    \item Every $\Delta$-QP-point is a QP-point of $B$. 
    \item If $\X$ is a Banach space, every flat point of $B$ is a $\Delta$-QP-point of $B$.
    \item\label{i: QP not rotund} If $B$ admits QP-points, then it is not rotund.
    \item\label{i: Gateaux QP is flat} Every G\^ateaux smooth QP-point of $B$ is a flat point.
\end{enumerate}
\end{lemma}

We shall also need the following simple fact, whose proof is a plain adaptation of the arguments used in \cite{DEVEtiling}*{Lemma~3.7}, \cite{KleeTri}*{Theorem~5.1}, and \cite{Fonf3poly}*{Theorem~2}. 

\begin{fact}\label{fact: QP} Let $\T$ be a convex tiling of a normed space $\X$, $B\in \T$, and $x\in \partial B$ be a regular point for $\T$. Then $x$ is a $\Delta$-QP-point of $B$. 
\end{fact}

\begin{proof} Let us denote by $\T_x=\{D\in\T\colon x\in D\}$; since $\T_x$ is finite, we can write 
\[ \T_x=\{B,B_1,\ldots,B_n\}. \]
As $x$ is a regular point, it is easy to see that $x\in\mathrm{int} (\bigcup\T_x)$ (see also \cite{Klee2}*{Theorem~1.1}). Let $V\subseteq \X$ be an open convex neighbourhood of $x$ such that $V\subseteq \mathrm{int} (\bigcup\T_x)$. For every $k=1,\ldots,n$, let $f_k\in \X^*$ be such that
\[ \sup f_k(B)=f_k(x)=\inf f_k(B_k). \]
We claim that \eqref{eq: LFC} holds, where the inclusion
\[ B\cap V\subseteq \bigcap_{k=1}^n \{z\in V\colon f_k(z)\leq f_k(x)\} \]
is clear. Suppose on the contrary that there exists $z\in V\setminus B$ such that $f_k(z)\leq f_k(x)$, whenever $k=1,\dots,n$. As $B$ is a body, we can take some $w\in V\cap\mathrm{int}(B)$; up to replacing $z$ with a point in $(z,w)\setminus B$, we can suppose that actually $f_k(z)< f_k(x)$, for all $k=1,\dots,n$. However, there exists $\overline k\in\{1,\dots,n\}$ such that $z\in B_{\overline k}$, which gives the contradiction that $f_{\overline k}(x)\leq f_{\overline k}(z)$.
\end{proof}

\begin{theorem}\label{thm: not rotund &co} Let $\nn\cdot$ be the equivalent norm on $\ell_2(\Gamma)$ constructed at the beginning of the present paragraph and $\T= \{V_d\}_{d\in\D}$ be the corresponding tiling. Then:
\begin{enumerate}
    \item\label{item: LUR} $B_{\nn\cdot}$ does not admit LUR points.
    \item\label{item: SE} If $x_0\in S_{\nn\cdot}$ is a strongly exposed point of $B_{\nn\cdot}$, then $\nn\cdot$ is not Fr\'echet differentiable at $x_0$.
    \item\label{item: locallyfinite} A point $x_0\in S_{\nn\cdot}$ is a regular point for $\T$ if and only if $\|x_0\|_2<1$. 
    \item\label{item: QP point dense} The set of $\Delta$-QP-points of $B_{\nn\cdot}$ is norm dense in $S_{\nn\cdot}$.
    \item\label{item: R} The norm  $\nn\cdot$ is not rotund.
    \item\label{item: G-smooth} The norm  $\nn\cdot$ is not G\^ateaux differentiable.
    \end{enumerate}
\end{theorem}

\begin{proof} \eqref{item: LUR} and \eqref{item: SE} immediately follow from \cite{DEVEtiling}*{Theorem~4.9, (ii), (iv)}, taking into account that $\T$ is a tiling by translates of the unit ball $B_{\nn\cdot}$.

Next, we prove \eqref{item: locallyfinite}. If $x_0\in S_{\nn\cdot}$ and $\|x_0\|_2=1$ then we necessarily have that $x_0$ is an extreme point of $B_{\nn\cdot}$ (since $B_{\nn\cdot}\subseteq B_{\n_2}$ and $B_{\n_2}$ is rotund). By \cite{FPZsingularpoints}*{Proposition~1}, $x_0$ is a singular point for $\T$. For the other implication, assume that $\|x_0\|_2<1$ and let $\e>0$ be such that $r\coloneqq\|x_0\|_2+2\e<1$. The same argument as in the proof of \Cref{lem: K(X) proximinal} ensures us that the set $S\coloneqq (x_0+r B_{\|\cdot\|_2})\cap \D$ is finite. Thus, for every $d\in \D \setminus S$ we have that $\|x_0- d\|_2 >r$. Hence, for $y\in \ell_2(\Gamma)$ with $\|x_0- y\|_2<\e$, we have 
\[ \|y\|_2< \|x_0\|_2+\e= r-\e< \|y- d\|_2. \]
By definition of the Voronoi cells, it follows that $y$ can only belong to the cells $\{V_d\}_{d\in S}$, and therefore $x_0$ is a regular point for $\T$.

To prove \eqref{item: QP point dense}, fix $x_0\in S_{\nn\cdot}$ and let us prove that arbitrarily close to $x_0$ we can find a $\Delta$-QP-point of $B_{\nn\cdot}$. If $\|x_0\|_2<1$ then, by \eqref{item: locallyfinite} and \Cref{fact: QP}, $x_0$ is a $\Delta$-QP-point of $B_{\nn\cdot}$, and we are done. So, we can suppose that $\|x_0\|_2=1$. Notice  that, by \eqref{item: LUR} and since $\n_2$ is LUR, every neighbourhood of $x_0$ contains a point $y_0$ such that $\|y_0\|_2< \nn{y_0}=1$. Arguing as above, $y_0$ is a $\Delta$-QP-point of $B_{\nn\cdot}$, and the proof of \eqref{item: QP point dense} is concluded.

The assertion in \eqref{item: R} immediately follows from \eqref{item: QP point dense} and \Cref{obs: flat VS QP}\eqref{i: QP not rotund}. Finally, we prove \eqref{item: G-smooth}. Since $B_{\nn\cdot}$ admits $\Delta$-QP-points (and hence QP-points), there exists a non-trivial segment $[x,y]\subseteq S_{\nn\cdot}$. Moreover, we can suppose that $[x,y]$ is maximal, in the sense that no other segment contained in $S_{\nn\cdot}$ contains properly $[x,y]$. In particular, $x$ is not a flat point of $B_{\nn\cdot}$. We claim that $x$ is not a  G\^ateaux smooth point of $B_{\nn\cdot}$. In fact, if $\|x\|_2<1$, then it is a QP-point (by \eqref{item: locallyfinite} and \Cref{fact: QP}, as in the previous paragraph); hence, it is not G\^ateaux smooth, by \Cref{obs: flat VS QP}\eqref{i: Gateaux QP is flat}. Thus, we can suppose that $\|x\|_2=1$. By the Hahn-Banach theorem, there exists $\p\in \ell_2(\Gamma)^*$ such that 
\[ \p(x)= \p(y)=1= \sup \p(B_{\nn\cdot})< \sup \p(B_{\|\cdot\|_2}), \]
where the inequality holds since, by strict convexity, $\|x+y\|_2<2$ and $\p(x+y)=2$. On the other hand, if we consider the functional $\psi \in\ell_2(\Gamma)^*$ defined by $\psi=\langle x,\cdot\rangle$, we have
\[ \psi(x)=1= \sup \psi(B_{\|\cdot\|_2})=\sup \psi(B_{\nn\cdot}). \]
Thus, $\p\neq \psi$, and $B_{\nn\cdot}$ admits two distinct support functionals at $x$, as desired.
\end{proof}

\subsection{Case \texorpdfstring{$p\in(1,\infty)$, $p\neq 2$}{1<p<∞}}\label{sec: p not 2}
We now briefly discuss the case when $p\in (1,\infty)$ is not necessarily equal to $2$. The considerations from the case when $p=2$ apply almost identically, the only difference being that when applying \Cref{prop: tile by Voronoi} we cannot ensure that the Voronoi cells are convex. On the other hand, when $p\in(1,\infty)$, $\ell_p(\Gamma)$ is LUR and therefore the Voronoi cells constitute a tiling by starshaped bodies. Hence, we have the following result.

\begin{theorem} For a cardinal $\Gamma$ with $\Gamma^\omega= \Gamma$ and $p\in (1,\infty)$, the Banach space $\ell_p(\Gamma)$ admits a point-countable lattice tiling by bounded and symmetric starshaped bodies.

Moreover, $\ell_p(\Gamma)$ admits a point-countable covering with balls of radius $1$.
\end{theorem}
This should be compared to the recent result \cites{DEGBnormal1, DEJSnormal2} that every separable Banach space admits a normal tiling by starshaped bodies (a tiling is \emph{normal} if there are constants $r\leq R$ such that every tile contains a ball of radius $r$ and is contained in a ball of radius $R$). Therefore, in the specific case of $\ell_p(\Gamma)$ we obtain a stronger conclusion, because lattice tilings are plainly normal; on the other hand, the results in \cites{DEGBnormal1, DEJSnormal2} apply to all separable Banach spaces, and not just to the non-separable spaces $\ell_p(\Gamma)$.

\begin{remark} Hilbert spaces are characterised by the convexity of Voronoi cells, \cite{Amir}*{(5.7) p.~42}. More precisely, if for all $x,y\in \X$ the Voronoi cells corresponding to $\D= \{x,y\}$ are convex, then $\X$ is a Hilbert space. This easily implies that, even in the case when $\D$ is a subgroup, the corresponding Voronoi cells might not be convex. Therefore, there seems to be no reason why the Voronoi cells corresponding to our construction should be convex. This, of course, doesn't exclude the possibility that for \emph{some} subgroup of $\ell_p(\Gamma)$ the corresponding Voronoi cells might be convex.
\end{remark}

\subsection{Case \texorpdfstring{$p=1$}{p=1}} In conclusion to this section we consider the case $p=1$ and we compare our construction to Klee's original construction from \cite{Klee1}. In particular, we prove \Cref{mainth: ell1 balls}. This case is even simpler and we do not need to consider Voronoi cells. In fact, the set $\D$ from \Cref{thm: ell_p subgroup} is $2$-separated and therefore the balls of radius $1$ centred at points of $\D$ have mutually disjoint interiors. Further, as $\D$ is $1$-dense, such balls cover $\ell_1(\Gamma)$. Therefore, the tiling $\{d+ B_{\ell_1(\Gamma)}\}_{d\in \D}$ is already the lattice tiling whose existence we claimed in \Cref{mainth: ell1 balls} (modulo the clause concerning point-$2$-finiteness, that we prove in \Cref{thm: only vertices touch} below).

At this point we have to compare our construction with the one by Klee. On the one hand, our construction has the additional feature of being a lattice tiling; on the other hand, Klee's tiling is disjoint. We now show the fact that these two constructions are indeed different: Klee's tiling cannot be made into a lattice tiling and our lattice tiling cannot be made to be disjoint.

\begin{proposition}\label{prop: lattice not disjoint} No lattice tiling by balls can be disjoint.
\end{proposition}

In the next section we shall expand on this result and deduce in particular that lattice tilings by balls cannot even be star-finite.

\begin{proof} Let $\D$ be the subgroup of $\X$ that generates the tiling. By definition, the tiling is obtained by translates of $B_\X$, hence every non-zero $d\in \D$ satisfies $\|d\|\geq 2$. Notice that there exists $d\in \D$ such that $d/2$ does not belong to $\D$. In fact, otherwise $2^{-k}d$ would belong to $\D$ for all $d\in \D$ and $k\in \N$, contradicting the fact that the group is discrete. 

Because we have a tiling, there exists $h\in \D$ such that $d/2\in h+ B_\X$, which directly implies that $\|d-2h\|\leq 2$. Therefore, the balls $d+ B_\X$ and $2h+ B_\X$ are not disjoint (and obviously $2h\in \D$). Finally, $d+ B_\X$ and $2h+ B_\X$ are distinct balls, because $d$ cannot be equal to $2h$, as $d/2\notin \D$.
\end{proof}

\begin{remark} Notice that we only used that $\D$ is a subgroup to ensure that $2h\in \D$ whenever $h\in\D$. Therefore, the above argument actually shows the following: for every disjoint tiling with balls of radius $1$ there is a center $d$ of some ball such that $2d$ is not the center of any ball in the tiling. 
\end{remark}

We now prove that, albeit not disjoint, the tiling of $\ell_1(\Gamma)$ from \Cref{mainth: ell1 balls} is as close as possible to being disjoint. In fact, not only each point belongs to at most $2$ tiles (\emph{i.e.}, the tiling is point-$2$-finite), but the only points that might belong to more than one tile are the extreme points of the tiles. In particular, if two tiles intersect, they only do in one point.

\begin{theorem} \label{thm: only vertices touch} Let $\D$ be the subgroup of $\ell_1(\Gamma)$ given by \Cref{thm: ell_p subgroup} (for $p=1$). Then 
\begin{equation} \label{eq: contained in the basis}
    \D\cap 2 S_{\ell_1(\Gamma)}\subseteq \{\pm 2e_\alpha\}_{\alpha<\Gamma}.
\end{equation}
In particular, the  tiling $\T=\{ d+B_{\ell_1(\Gamma)} \}_{d\in\D}$ from \Cref{mainth: ell1 balls} is point-$2$-finite and the only points of $B_{\ell_1(\Gamma)}$ that can belong to another tile are $\{\pm e_\alpha\}_{\alpha<\Gamma}$.
\end{theorem}

\begin{remark} In \Cref{sec: intersections} we will prove that every tile of a lattice tiling of $\X$ by balls intersects $\dens(\X)$ other tiles (\Cref{cor: lattice not star-finite}). Combining this fact and \Cref{thm: only vertices touch}, we see that the subgroup $\D$ of $\ell_1(\Gamma)$ contains `many' vectors of the canonical basis:
\[ \big| \D\cap \{\pm 2e_\alpha\}_{\alpha<\Gamma} \big|= |\Gamma|. \]
\end{remark}

\begin{proof} Let $u_\alpha$ and $\D_\alpha$ ($\alpha<\Gamma$) be defined as in the proof of \Cref{thm: ell_p subgroup}. Since $\D= \bigcup_{\alpha<\Gamma} \D_\alpha$, it is clearly sufficient to prove that, for every $\alpha<\Gamma$,
\begin{equation}\label{eq: only vertices touch}
    \D_\alpha\cap 2 S_{\ell_1(\Gamma)}\subseteq \{\pm 2e_\beta\}_{\beta<\Gamma}.
\end{equation}

The validity of \eqref{eq: only vertices touch} is checked by transfinite induction, following the argument in \Cref{thm: ell_p subgroup}. Suppose that, for some $\gamma< \Gamma$, \eqref{eq: only vertices touch} holds for all $\alpha< \gamma$. If we are in the case when $\D_\gamma= \bigcup_{\alpha<\gamma} \D_\alpha$, then \eqref{eq: only vertices touch} trivially also holds for $\gamma$. Otherwise, every $x\in \D_\gamma$ can be written as $x= d+ n(u_\gamma+ e_{\tilde{\gamma}})$. Further, by \eqref{eq: split norm} we have that
\begin{equation}\label{eq: split norm vertices}
    \|x\|_1= \|d+ n u_\gamma\|_1+ |n|.
\end{equation}
The assumption that $\|x\|_1=2$ then gives $|n|\leq 2$. If $n=0$, then $x=d\in \bigcup_{\alpha< \gamma} \D_\alpha$ and the conclusion follows by the transfinite induction assumption. If $|n|=1$, then $\|d \pm u_\gamma\|_1>1$ by construction, and $\|x\|_1>2$. Thus, we must have that $|n|=2$. But then, \eqref{eq: split norm vertices} yields $d+ n u_\gamma=0$, whence $x=\pm2 e_{\tilde{\gamma}}$.

For the second part, it is sufficient to observe that, if $x\in B_{\ell_1(\Gamma)}$ and there exists $d\in\D\setminus\{0\}$ such that $x\in d+B_{\ell_1(\Gamma)}$, then $\|d\|_1=2$; thus, $d=\pm 2e_\alpha$ for some $\alpha<\Gamma$, by \eqref{eq: contained in the basis}. Hence, $x= \pm e_\alpha$ and $d=2x$, which shows the uniqueness of $d$.
\end{proof}

We conclude this section with a generalisation of \Cref{mainth: ell1 balls}, under renorming. The construction by Klee actually works for all normed spaces of the form $\X\oplus_1 \ell_1(\Gamma)$ under the assumptions that $\Gamma^\omega= \Gamma$ and $\dens(\X)\leq \Gamma$, \cite{Klee1}*{Theorem 1.2}. In particular, if a normed space $\X$ of density $\Gamma$ contains a complemented copy of $\ell_1(\Gamma)$, there is an equivalent norm on $\X$ in which it admits a disjoint tiling with balls of radius $1$. Similarly, a modification of our proof in the spirit of \cite{Klee1}*{Theorem 1.2} gives a lattice tiling with balls. We now explain how to prove the same results without the complementation assumption on the subspace $\ell_1(\Gamma)$. This depends on recent results on octahedrality, \cite{AMR_ell1(kappa)}, and we are most grateful to Esteban Mart\'inez Va\~{n}\'o and Abraham Rueda Zoca for suggesting the result below and pointing to our attention the relevance of \cite{AMR_ell1(kappa)}.

The main result in \cite{AMR_ell1(kappa)} is that a Banach space $\X$ contains an isomorphic copy of $\ell_1(\Gamma)$ if and only if there exists an equivalent norm $\nn\cdot$ on $\X$ such that for every subspace $\Y$ of $\X$ with $\dens(\Y)< \Gamma$ there exists a unit vector $x$ such that
\begin{equation}\label{eq: renorm when ell1 inside}
    \nn{y+ rx} = \nn{y} + |r|\qquad \mbox{ for all } y\in \Y \mbox{ and } r\in \R.
\end{equation}

\begin{theorem}\label{thm: tiling if ell1 inside} Let $\Gamma$ be a cardinal number with $\Gamma^\omega= \Gamma$ and $\X$ be a Banach space of density $\Gamma$. If $\X$ contains an isomorphic copy of $\ell_1(\Gamma)$, then there is an equivalent norm $\nn\cdot$ on $\X$ with the following properties:
\begin{enumerate}
    \item\label{i: disjoint oct} $(\X,\nn\cdot)$ admits a disjoint tiling by balls of radius $1$;
    \item\label{i: lattice oct} $(\X,\nn\cdot)$ admits a lattice tiling by balls.
\end{enumerate}
\end{theorem}
Needless to say, the advantage of \Cref{mainth: ell1 balls} over the present one is that \Cref{mainth: ell1 balls} holds for the original norm.

\begin{proof} We explain the proof of \eqref{i: lattice oct}; the argument for \eqref{i: disjoint oct} is similar and only requires replacing our argument as in \Cref{mainth: ell1 balls} with Klee's argument in \cite{Klee1}*{Theorem 1.2}. By \cite{AMR_ell1(kappa)}, there exists a norm $\nn\cdot$ on $\X$ such that for every subspace $\Y$ of $\X$ with $\dens(\Y)< \Gamma$ there exists $x\in \Y$ such that \eqref{eq: renorm when ell1 inside} holds. Crucially, the only point in the proof of \Cref{thm: ell_p subgroup} where the $\ell_p$-norm was used is \eqref{eq: split norm}; however, it is clear that the validity of \eqref{eq: renorm when ell1 inside} is enough for performing the computation in \eqref{eq: split norm}. Therefore, the very same argument in the proof of \Cref{thm: ell_p subgroup} carries over to prove that $(\X,\nn\cdot)$ contains a $2$-separated and $1$-dense subgroup. As for \Cref{mainth: ell1 balls}, \eqref{i: lattice oct} immediately follows.
\end{proof}

\section{Lattice tilings are not star-finite}\label{sec: intersections}
In  this section we generalise the observation that lattice tilings are not disjoint (\Cref{prop: lattice not disjoint}), by investigating how many tiles must each tile of a lattice tiling intersect. In the main result of the section, \Cref{cor: lattice not star-finite} below, we show that this cardinality is as large as possible (\emph{i.e.}, equal to the cardinality of the tiling). Our main tool will be the following more general result.

\begin{theorem}\label{thm: 2theta intersects a lot}
Let $\D$ be a subgroup of an infinite-dimensional normed space $\X$ of density $\Gamma$. If $\D$ is $r$-dense, then $|\D\cap 2r B_\X|\geq\Gamma$.
\end{theorem}

\begin{proof} Up to a scaling, we can assume without any loss of generality that $r=1$. Suppose, towards a contradiction, that $|\D\cap 2 B_\X|< \Gamma$. We begin by observing that, for every infinite cardinal $\Gamma_0$ with $|\D\cap 2 B_\X|\leq \Gamma_0\leq \Gamma$ and for every subspace $\Z$ of $\X$ with $\dens(\Z)\leq \Gamma_0$ we have $|\D\cap \Z|\leq \Gamma_0$. In fact, if $|\D\cap \Z|> \Gamma_0$, there would be a ball of radius $1$ in $\Z$ that contains more than $\Gamma_0$ elements of $\D\cap \Z$. This ball is contained in a ball in $\Z$ of radius $2$ and with center $d\in \D$; hence, $|\D\cap B_2(d)|>\Gamma_0$. Because $\D$ is a subgroup, this implies $|\D\cap 2 B_\X|> \Gamma_0$, a contradiction.

We now distinguish two cases.
\smallskip

\noindent{\bf Case 1: $\Gamma=\omega$.} This case is a finite-dimensional reduction based on Sierpi\'nski's theorem on continua, \cite{Sierpinski}. First of all, taking $\Gamma_0= \Gamma$ in the previous observation, we see that $\D$ is countable. Further, the subgroup $\D_0$ generated by $\D\cap 2 B_{\X}$ is finitely generated. Therefore, $\G=\D/\D_0$ is a countable infinite group. For each $g=[d]=d+\D_0\in \G$, we can consider the closed subset of $\X$ defined as
\[ P_{g}=d+\D_0+B_{\X}= \bigcup_{[d]=g} d+B_\X. \]
Because $\D_0$ is finitely generated, the above union is locally finite; hence, $P_g$ is indeed closed. The crucial point is that $\|d_1 - d_2\|>2$, whenever $[d_1],[d_2]\in \G$ are different. This and the fact that $\D$ is 1-dense yield that:
\begin{enumerate}
    \item\label{i: Pg disjoint} $P_{g_1}\cap P_{g_2}=\emptyset$, for distinct $g_1,g_2\in \G$;
    \item\label{i: Pg cover} $\bigcup_{g\in \G}P_{g}=\X$.
\end{enumerate}
Now take two points that belong to distinct sets $P_g$ and let $\ell$ be the segment joining them. By \eqref{i: Pg disjoint} and \eqref{i: Pg cover} we can write $\ell$ as the disjoint countable union $\bigcup_{g\in \G}(P_{g}\cap \ell)$, where each $P_{g}\cap \ell$ is closed and at least two are non-empty. However, this is impossible by Sierpi\'nski's result \cite{Sierpinski} (\emph{cf.}\ \cite{Kuratowski}*{\S 47, III, Theorem 6}).
\smallskip 

\noindent{\bf Case 2: $\Gamma>\omega$.} Let $\Z$ be the closed span of $\D\cap 2 B_\X$ and $\Gamma_0\coloneqq \max\{|\D\cap 2 B_\X|, \omega\}$. Then, $\dens(\Z)\leq \Gamma_0$ and, as $\Gamma$ is uncountable, $\Gamma_0<\Gamma$. The observation at the beginning of the proof then gives $|\D\cap \Z|\leq \Gamma_0$. Let us denote by $q\colon \X\to \X/\Z$ the canonical quotient map and $\H\coloneqq q(\D)$. Clearly, $\H$ is a subgroup of $\X/\Z$.
\begin{claim}\label{claim: H=2H and others} We claim that:
\begin{enumerate}[label=\normalfont{(C\arabic*)},ref=C\arabic*]
    \item\label{i: claim big card} $|\H|=\Gamma$; 
    \item\label{i: claim small card} the set $\V\coloneqq \{h\in \H\colon\, \|h\|_{\X/\Z}<1/3\}$ has cardinality at most $\Gamma_0$;
    \item\label{i: claim H=2H} the group $\H$ satisfies $\H=2\H$.
\end{enumerate}
\end{claim} 

\begin{proof}[Proof of \Cref{claim: H=2H and others}]\renewcommand\qedsymbol{$\square$} The validity of \eqref{i: claim big card} is clear, because $|\D|\geq \Gamma$, while $|\D\cap \Z|\leq \Gamma_0< \Gamma$.

Let us now prove \eqref{i: claim small card}. Suppose that $|\V|\eqqcolon \Gamma'>\Gamma_0$ and take a family $(d_\alpha+\Z)_{\alpha<\Gamma'}$ of mutually distinct elements of $\V$; then, there exists a family $(z_\alpha)_{\alpha<\Gamma'}\subseteq \Z$ such that $\|d_\alpha+z_\alpha\|<1/3$ ($\alpha<\Gamma'$). Since $\dens(\Z)< \Gamma'$, up to passing to a subfamily still of cardinality $\Gamma'$, we can assume that $\|z_{\alpha_1}-z_{\alpha_2}\|<1/3$ for all distinct $\alpha_1,\alpha_2\in \Gamma'$. Thus
\[ \|d_{\alpha_1}-d_{\alpha_2}\|\leq \|d_{\alpha_1}+z_{\alpha_1}\|+\|d_{\alpha_2}+z_{\alpha_2}\|+\|z_{\alpha_1}-z_{\alpha_2}\|<1; \]
hence, $(d_\alpha -d_0)_{\alpha< \Gamma'}\subseteq \D \cap B_\X$, a contradiction.

It remains to prove \eqref{i: claim H=2H}. For this, it is sufficient to prove that $\D\subseteq 2\D+ \Z$. Indeed, if this is the case, we have
\[ \H= q(\D)\subseteq 2q(\D)+ q(\Z)= 2\H. \]
Thus, take any $d\in \D$. Since $\D$ is 1-dense, there exists $g\in \D$ such that $\|d/2-g\|\leq 1$, whence $\|d-2g\|\leq 2$. By definition of $\Z$, this means that $d-2g\in \Z$. But then, $d\in 2\D+ \Z$, which concludes the proof of our claim.
\end{proof}
 
Now, by \eqref{i: claim H=2H}, for every $h\in \H$ we have that $h/2\in \H$ and iteratively $2^{-n}h\in \H$. Thus,
\[ \H=\bigcup_{n\in\N}2^n \V, \]
which implies $|\H|\leq \Gamma_0$ because of \eqref{i: claim small card} and the fact that $\Gamma_0$ is infinite. This contradicts \eqref{i: claim big card}, and the conclusion holds.
\end{proof}

We now study intersections of $r$-dense subgroups with balls of radius smaller than $2r$, the natural scaling factor being given by the Kottman constant. On the other hand, we obtain intersections that, in general, have a smaller cardinality. We shall need the following simple fact.

\begin{fact}\label{fact: ramsey} Let $\W$ be an infinite subset of a normed space $\X$  and $r>0$. If
\[ |B_r(w)\cap \W| <|\W| \]
for all $w\in \W$, then $\W$ contains an infinite $(r+)$-separated set.    
\end{fact}

\begin{proof} Take any $w_1\in \W$; by assumption, there is $w_2\in \W\setminus B_r(w_1)$. Inductively, choose $w_{n+1}\in \W\setminus (B_r(w_1)\cup\dots \cup B_r(w_n))$. The sequence $(w_n)_{n=1}^\infty$ is clearly $(r+)$-separated.
\end{proof}

\begin{proposition}\label{prop: singularpointrick} Let $\X$ be an infinite-dimensional Banach space and $\D$ be a 1-dense subgroup of $\X$. Then, for every $\e>0$, the following assertions hold:
\begin{enumerate}
    \item\label{eq: singular} the set $\D\cap \bigl(K(\X)+\e\bigr)B_\X$ is infinite;
    \item\label{eq: Gamma-singular} if $\X$ is a Hilbert space and $\D$ is $\sqrt{2}$-separated, then $|\D\cap (\sqrt 2+\e)B_\X|=\dens(\X)$.
\end{enumerate}     
\end{proposition}

\begin{proof} We begin with the proof of \eqref{eq: singular}, that is inspired by \cite{CPZ}*{Theorem 1.1}. If $\X$ contains an isomorphic copy of $c_0$, then $K(\X)=2$ and our conclusion follows from \Cref{thm: 2theta intersects a lot}. Hence, we assume that $\X$ does not contain  isomorphic copies of $c_0$. Fix $\e>0$ and take $\delta>0$ such that $\frac{K(\X)+\e}{1+\delta}\geq K(\X)+\frac{\e}{2}$. As $\X$ does not contain copies of $c_0$, the covering $\B\coloneqq \{d+B_{\X}\}_{d\in \D}$ of $\X$ admits a singular point $x_0$, by \cite{FZ_PAMS06}*{Corollary~5}. Then the ball $B_\delta(x_0)$ intersects infinitely many different members of $\mathcal B$; thus, the set $\W_\delta\coloneqq B_{1+\delta}(x_0)\cap \D$ is infinite.

Now, suppose towards a contradiction that $\bigl|\D\cap \bigl(K(\X)+\e\bigr)B_\X\bigr|=n\in \N$ and observe that, since $\D$ is a subgroup of $\X$, we have in particular
\[ \big| \W_\delta\cap B_{K(\X)+\e}(w)\big|\leq n,\qquad \mbox{for all } w\in \W_\delta. \]
An application of \Cref{fact: ramsey} yields a $\bigl(K(\X)+\e\bigr)$-separated sequence $(d_n)_{n=1}^\infty\subseteq \W_\delta$. Therefore, the sequence $\left(\frac{d_n-x_0}{1+\delta}\right)_{n=1}^\infty\subseteq B_\X$ is $\bigl(K(\X)+\frac{\e}{2}\bigr)$-separated, which contradicts the definition of $K(\X)$. 

In order to prove \eqref{eq: Gamma-singular} we proceed similarly, taking into account that $K(\ell_2(\Gamma))=\sqrt{2}$. By \Cref{lem: K(X) proximinal}, the set $\D$ is automatically proximinal and hence, by \Cref{prop: tile by Voronoi}, the family $\T\coloneqq \{V_d\}_{d\in\D}$ consisting of the associated Voronoi cells is a tiling of $\ell_2(\Gamma)$ by bounded convex bodies. By the Krein--Milman theorem the (weakly compact) convex body $V_0$ admits an extreme point $x_0$; by \cite{FPZsingularpoints}*{Proposition~1}, $x_0$ is a $\Gamma$-singular point for $\T$. Thus, for every $\delta>0$, the ball $B_\delta(x_0)$ intersects $\Gamma$-many different members of $\T$. Since $V_d\subseteq d+B_\X$, whenever $d\in\D$, we necessarily have $\bigl|B_{1+\delta}(x_0)\cap \D\bigr|=\Gamma$ and we can proceed as in \eqref{eq: singular}. Let $\W_\delta \coloneqq B_{1+\delta}(x_0)\cap \D$ and suppose on the contrary that $\bigl|\D\cap \bigl(K(\X)+\e\bigr)B_\X\bigr|\eqqcolon \Gamma_0<\Gamma$. Since $\D$ is a subgroup of $\X$, we have
\[ \big| \W_\delta\cap B_{K(\X)+\e}(w) \big|\leq \Gamma_0, \qquad \mbox{for all } w\in \W_\delta. \]
Finally, \Cref{fact: ramsey} gives a contradiction as in the proof of \eqref{eq: singular}. 
\end{proof}

We now turn to rephrasing the previous results in terms of lattice tilings by balls. As an immediate consequence of \Cref{thm: 2theta intersects a lot} we get the following.
\begin{corollary}\label{cor: lattice not star-finite} Let $\T$ be a lattice tiling by balls of an infinite-dimensional normed space $\X$. Then, for every $T_0\in \T$, we have
\[ \big| \{T\in \T\colon T\cap T_0 \neq \emptyset\} \big|= |\T|. \]
In particular, no lattice tiling of $\X$ with balls is star-finite.    
\end{corollary}

In conclusion to this section we give an alternative shorter proof of \Cref{cor: lattice not star-finite}, which uses the fact that discrete subgroups of normed spaces are free (Abelian). For countable subgroups, this is a folklore result that can be found, \emph{e.g.}, in \cites{ADG, Lawrence, Sidney, Zorzitto}. The general case is due to Stepr\={a}ns \cite{Steprans}. We refer to \Cref{sec: subgroup} for more details on free Abelian groups and for an alternative proof of Stepr\={a}ns' result.

\begin{proof}[Second proof of \Cref{cor: lattice not star-finite}] Clearly, it is enough to prove the claim for $T_0$ being the ball centred at $0$. The group $\D$ that generates the tiling is obviously $2$-separated, hence free by Stepr\={a}ns' result mentioned above. Thus, we can fix a basis $\Lambda$ for $\D$. Our aim is to prove that
\[ \{d\in \D\colon \|d\|\leq 2\} \]
has cardinality $\Gamma\coloneqq |\T|$ (which clearly coincides with $|\D|$).

Towards a contradiction, suppose that this is false. Every element $d\in \D$ with $\|d\|\leq 2$ can be written as a finite (integer) linear combination of elements of $\Lambda$; hence, there exists a subset $\Lambda_0$ of $\Lambda$ of cardinality less than $\Gamma$ and such that the group $\G_0$ generated by $\Lambda_0$ contains $\{d\in \D\colon \|d\|\leq 2\}$. The counterpart to \Cref{claim: H=2H and others} is the following claim.
\begin{claim}\label{claim: usual 2D stuff} $\D\subseteq 2\D+ \G_0$, hence $\D/\G_0= 2\D/\G_0$.
\end{claim}

\begin{proof}[Proof of \Cref{claim: usual 2D stuff}] \renewcommand\qedsymbol{$\square$} This is essentially identical to the proof of \eqref{i: claim H=2H}. Since $\T$ is a tiling, for any $d\in \D$ there exists $h\in \D$ such that $\|d/2-h\|\leq 1$, or $\|d-2h\|\leq 2$. By definition, this means that $d-2h\in \G_0$, whence $d\in 2\D+ \G_0$. The second part also follows as above.
\end{proof}

The main simplification comes now. Because $\G_0$ is generated by the subset $\Lambda_0$ of the basis $\Lambda$, the group $\D/\G_0$ is canonically isomorphic to the subgroup $\H$ of $\D$ generated by $\Lambda\setminus \Lambda_0$. \Cref{claim: usual 2D stuff} then gives that $\H= 2\H$. Thus, for every $h\in \H$ and $n\in\N$ we have that $2^{-n}h\in \H$. Because $\H$ is a subgroup of the discrete group $\D$, this is only possible if $\H=0$. Consequently, it follows that $\Lambda=\Lambda_0$, in which case we get that $|\Lambda|< \Gamma$ and thus $|\D|<\Gamma$, a contradiction.
\end{proof}

\section{Discrete subgroups of normed spaces}\label{sec: subgroup}
In this section we depart from our main interest of lattice tilings and we consider discrete subgroups of normed spaces, without aiming at explicit applications to tilings. The main results here are the proof of \Cref{mainth: subgroups}\eqref{mth: subgroup general} and of Stepr\={a}ns' theorem that we mentioned in the previous section. We begin with a simple lemma that allows us to automatically obtain a bounded set of generators in an $r$-dense subgroup.

\begin{lemma}\label{lem: bdd generator} Let $\D$ be an $r$-dense subgroup of a normed space $\X$. Then, for every $\e>0$, $\D$ is generated by the set $\D\cap (2r+\e)B_\X$.
\end{lemma}

\begin{proof} Let us define $\Lambda\coloneqq \D\cap (2r+\e)B_\X$ and let $\D_0$ be the group generated by $\Lambda$. We shall prove that $\D_0$ coincides with $\D$. We first observe that $(r+\e)B_\X\subseteq \Lambda+ rB_\X$. In fact, if $\|x\|\leq r+\e$, we can find $d\in \D$ such that $\|x- d\|\leq r$; clearly, $\|d\|\leq 2r+\e$, whence $d\in \Lambda$. This yields $x\in \Lambda+ rB_\X$ and proves the inclusion. In turn, this implies
\begin{equation*} \begin{split}
    (r+2\e)B_\X &= (r+\e) B_\X + \e B_\X \subseteq \Lambda + (r+\e)B_\X\\ &
    \subseteq \Lambda+\Lambda + rB_\X \subseteq \D_0 + rB_\X.
\end{split}
\end{equation*}
Iterating the same argument we obtain $(r+n\e)B_\X\subseteq \D_0 + rB_\X$, and thus $\X\subseteq \D_0+ rB_\X$. Finally, we show that $\D_0$ coincides with $\D$. Let $d\in \D$. Since $\X\subseteq \D_0+ rB_\X$, there exists $h\in \D_0$ such that $\|d-h\|\leq r$. By definition, $d-h \in \Lambda\subseteq \D_0$, which implies $d\in\D_0$.
\end{proof}

\begin{remark} In general it is not possible to take $\e=0$ in \Cref{lem: bdd generator}. In fact, in \Cref{thm: only vertices touch} we proved that the $1$-dense subgroup $\D$ of $\ell_1(\Gamma)$ from \Cref{mainth: subgroups}\eqref{mth: subgroup ell_p} satisfies 
\[ \D\cap 2B_{\ell_1(\Gamma)}\subseteq \{0, \pm2 e_\alpha\}_{\alpha<\Gamma}. \]
Therefore, the subgroup generated by $\D\cap 2B_{\ell_1(\Gamma)}$ is a subgroup of the even integer grid $\bigoplus_{\alpha<\Gamma} (2e_\alpha)\ZZ$, which is clearly not $1$-dense. Hence, $\D$ is not generated by $\D\cap 2B_{\ell_1(\Gamma)}$.
\end{remark}

We are now ready for the proof of \Cref{mainth: subgroups}\eqref{mth: subgroup general}, whose statement we repeat here, for convenience of the reader.

\begin{theorem}\label{thm: group in X} Let $\X$ be an infinite-dimensional normed space. Then, for every $\e>0$, there exists a subgroup $\D$ of $\X$ that is $1$-separated and $(1+\e)$-dense. Further, the group is generated by a set of vectors of norm at most $2+\e$.
\end{theorem}

\begin{proof} Let $\Gamma$ be the density character of $\X$ and select a collection $(u_\alpha)_{\alpha< \Gamma}$ of vectors in $\X$ that is dense in $\X$; we also assume that $u_0=0$. We now build by transfinite induction an increasing chain $(\D_\alpha)_{\alpha< \Gamma}$ of subgroups of $\X$ such that for all $\alpha<\Gamma$
\begin{enumerate}
    \item\label{5.2.i: septd} $\D_\alpha$ is $1$-separated,
    \item\label{5.2.ii: gen} $\D_\alpha$ is generated by at most $|\alpha|$ elements,
    \item\label{5.2.iii: close} there is $d\in \D_\alpha$ with $\|u_\alpha- d\|\leq 1+\e$.
\end{enumerate}
This directly implies the result. In fact, $\D\coloneqq \bigcup_{\alpha< \Gamma} \D_\alpha$ is clearly a $1$-separated subgroup of $\X$. Further, given $x\in \X$ there is $\alpha< \Gamma$ with $\|x- u_\alpha\|\leq \e$. Therefore, there is $d\in \D_\alpha$ with $\|x-d\|\leq 1+2\e$, and $\D$ is $(1+2\e)$-dense. Finally, it follows at once from \Cref{lem: bdd generator} that $\D$ is generated by the vectors of norm at most $2+3\e$.

The transfinite induction argument is a simplification of the one in \Cref{thm: ell_p subgroup}. We begin by setting $\D_0=\{0\}$. Suppose that we have already defined the desired subgroups $(\D_\alpha)_{\alpha< \gamma}$, for some $\gamma< \Gamma$. Then, $\bigcup_{\alpha< \gamma} \D_\alpha$ is a subgroup of $\X$ that is $1$-separated and is generated by at most $|\gamma|$ elements (in case when $\gamma= \beta+1$ is a finite ordinal, $\bigcup_{\alpha< \gamma} \D_\alpha= \D_\beta$, so the claim also holds). In particular, $\Z_\gamma\coloneqq \overline{\rm span}(\bigcup_{\alpha< \gamma} \D_\alpha\cup \{u_\gamma\})$ is a proper subspace of $\X$. Therefore, we are in position to apply Riesz' lemma and obtain a norm one vector in $\X$ that has distance at least $(1+\e)^{-1}$ from $\Z_\gamma$. By rescaling, we thus can  obtain a vector $x_\gamma\in \X$ such that $\|x_\gamma\|\leq 1+\e$ and $\dist(x_\gamma,\Z_\gamma)\geq 1$. We can now define
\[ \D_\gamma\coloneqq \bigcup_{\alpha< \gamma} \D_\alpha+ (u_\gamma+ x_\gamma)\ZZ. \]

By definition, $\D_\gamma$ is generated by adding just one element to the generators of $\bigcup_{\alpha< \gamma} \D_\alpha$; thus, \eqref{5.2.ii: gen} holds. \eqref{5.2.iii: close} is trivial because $u_\gamma+ x_\gamma$ belongs to $\D_\gamma$ and its distance from $u_\gamma$ is at most $1+\e$. Finally, to check the validity of \eqref{5.2.i: septd}, we check that every non-zero element in $\D_\gamma$ has norm at least $1$. Thus, take a vector $d+n(u_\gamma+ x_\gamma)$, where $d\in \bigcup_{\alpha< \gamma} \D_\alpha$ and $n\in \ZZ$. In the case when $n=0$, the norm of $d+n(u_\gamma+ x_\gamma)$ is $\|d\|\geq1$. Instead, if $n\neq 0$,
\[ \|d+n(u_\gamma+ x_\gamma)\|= |n| \left\|\left(\frac{d}{n}+ u_\gamma\right) +x_\gamma \right\|. \]
By definition, $-(\frac{d}{n}+ u_\gamma)$ belongs to $\Z_\gamma$, so it has distance at least $1$ from $x_\gamma$. Hence, the term above is at least $|n|\geq1$, and we are done.
\end{proof}

\begin{remark} With a small adaptation of the proof, we can also obtain a subgroup $\D$ that is $1$-separated and with $\dist(x,\D)\leq 1$ for every $x\in \X$ (namely, $\D$ is $(1+\e)$-dense for every $\e>0$). In fact, when $\X$ is separable it is just enough to use $\e_n= 2^{-n}$ in the $n$-th step of the argument. In the non-separable case, instead, we have to replace \eqref{5.2.ii: gen} with $|\D_\alpha|\leq \max \{|\alpha|,\omega\}$ in the inductive construction. Then, having the subspace $\Z_\gamma$, we find a vector $x_\gamma^1$ of norm at most $1+\e_1$ and having distance $1$ from $\Z_\gamma$. We then repeat the argument with $\Z_\gamma^1\coloneqq \spn\{\Z_\gamma, x_\gamma^1\}$ and find, at distance $1$ from $\Z_\gamma^1$, a vector $x_\gamma^2$ of norm at most $1+\e_2$. We continue by induction in the same way and we then let $\D_\gamma$ be the group generated by $\bigcup_{\alpha< \gamma} \D_\alpha$ and the vectors $\{u_\gamma+ x_\gamma^n\}_{n\in \N}$.
\end{remark}

We now pass to the second part of the section where we give an alternative proof of Stepr\={a}ns' result \cite{Steprans} (see also \cite{Fuchs}*{Section 3.10}) that discrete subgroups of normed spaces are free, which we used in the previous section. The same result also appears in the recent paper \cite{KaniaKostana}. Both proofs in \cite{Steprans} and \cite{KaniaKostana} depend upon Shelah's singular compactness theorem \cite{Shelah} and set theoretical machinery, such as the pressing-down lemma or elementary submodels, respectively. Our argument below is essentially a translation of \cite{KaniaKostana} into a transfinite induction proof. The main advantage of this more tangible approach is that it makes apparent that no regularity is required in the proof; as a consequence, our argument does not require the singular compactness theorem from \cite{Shelah} and is entirely self-contained.

Recall that an Abelian group $\G$ is \emph{free} if it admits a \emph{basis}, namely a subset $\Lambda$ such that every element $g\in \G$ can be uniquely expressed as
\[ g= \sum_{j=1}^n k_j \lambda_j, \quad \mbox{where } k_j\in \ZZ \mbox{ and } \lambda_j\in \Lambda. \]
It is an easy exercise to check that being free is a three-space property. More precisely, if $\H$ is a free subgroup of an Abelian group $\G$ and $\G/ \H$ is free, then $\G$ is also free and every basis of $\H$ can be extended to a basis of $\G$. By iterating this, one obtains the following standard fact (see, \emph{e.g.}, \cite{Eklof}*{Theorem 2.6}, or \cite{Fuchs}*{Section 3.7}).

\begin{fact}\label{fact: free by chain} Suppose that $\G$ is an Abelian group and $(\G_\alpha)_{\alpha< \Gamma}$ is a continuous increasing chain of subgroups\footnote{In group theory literature this is usually called a \emph{filtration}.} such that $\G= \bigcup_{\alpha< \Gamma} \G_\alpha$. If $\G_\alpha$ and $\G_{\alpha+1}/ \G_\alpha$ are free for each $\alpha< \Gamma$, then $\G$ is also free.
\end{fact}
By continuity of the chain we mean that $\G_\gamma= \bigcup_{\alpha< \gamma}\G_\alpha$, when $\gamma< \Gamma$ is a limit ordinal.

\begin{proof} We construct by transfinite induction a continuous increasing chain $(\Lambda_\alpha)_{\alpha< \Gamma}$ such that each $\Lambda_\alpha$ is a basis for $\G_\alpha$. Start with a basis $\Lambda_0$ of $\G_0$. At successor stages, we use the three-space property and the fact that $\G_{\alpha+1}/ \G_\alpha$ is free to extend the basis $\Lambda_\alpha$ of $\G_\alpha$ to a basis $\Lambda_{\alpha+1}$ of $\G_{\alpha+1}$. At a limit stage $\gamma$, we set $\Lambda_\gamma\coloneqq \bigcup_{\alpha< \gamma}\Lambda_\alpha$. By continuity of the chain $(\G_\alpha)_{\alpha< \Gamma}$, we see that $\Lambda_\gamma$ is a basis for $\G_\gamma$, which concludes the construction. By the same reasoning, it is then clear that $\bigcup_{\alpha< \Gamma}\Lambda_\alpha$ is a basis for $\G$.
\end{proof}

\begin{theorem}[Stepr\={a}ns, \cites{Steprans}]\label{thm: discrete is free} Discrete subgroups of normed spaces are free.
\end{theorem}

\begin{proof} We argue by transfinite induction on the cardinality $\Gamma$ of the group $\G$. The case when $\Gamma= \omega$ is a folklore fact \cites{ADG, Lawrence, Sidney, Zorzitto}, that also follows from a simplification of the argument below. Therefore, we fix a cardinal $\Gamma> \omega$ and we assume that all discrete subgroups $\G$ of normed spaces, with $|\G|< \Gamma$, are free. We now fix a discrete subgroup $\G$ of a normed space $\X$ such that $|\G|= \Gamma$; without loss of generality, assume that $\G$ is $1$-separated.

The following claim is at the core of the argument and it is essentially \cite{KaniaKostana}*{Lemma 3.1}.
\begin{claim}\label{claim: closing-off} For every subspace $\Z$ of $\X$ there exists a subspace $\tilde{\Z}\supseteq \Z$ of $\X$ such that $\dens(\Z)= \dens(\tilde{\Z})$ and with the property that, for $g\in \G$,
\begin{equation*}\label{eq: dagger}\tag{$\dagger$}
    \dist(g, \tilde{\Z})< 1/3 \quad \Longrightarrow \quad g\in \tilde{\Z}.  
\end{equation*}
\end{claim}
\begin{proof}[Proof of \Cref{claim: closing-off}] \renewcommand\qedsymbol{$\square$} Let $\Z_0\coloneqq \Z$, consider the set $\Lambda_0\coloneqq \{g\in \G\colon \dist(g, \Z_0)< 1/3\}$, and define $\Z_1\coloneqq \overline{\spn}\{ \Z_0, \Lambda_0\}$. Next, take $\Lambda_1\coloneqq \{g\in \G\colon \dist(g, \Z_1)< 1/3\}$ and let $\Z_2\coloneqq \overline{\spn}\{ \Z_1, \Lambda_1\}$. We continue by induction in the obvious way and we then define $\tilde{\Z}\coloneqq \overline{\bigcup_{k<\omega} \Z_k}$. The validity of \eqref{eq: dagger} is clear: if $\dist(g, \tilde{\Z})< 1/3$, then $\dist(g, \Z_k)< 1/3$ for some $k$, whence $g\in \Z_{k+1}\subseteq \tilde{\Z}$.

Finally, we check that $|\Lambda_k|\leq \dens(\Z_k)$ for each $k< \omega$; this yields $\dens(\Z_{k+1})= \dens(\Z_k)$, whence $\dens(\Z)= \dens(\tilde{\Z})$. By definition of $\Lambda_k$, for each $g\in \Lambda_k$ there exists $z_g\in \Z_k$ such that $\|g- z_g\|<1/3$. If it were that $|\Lambda_k|> \dens(\Z_k)$, it would follow that there are distinct $g,h\in \Lambda_k$ such that $\|z_g- z_h\|<1/3$, which would yield the contradiction that $\|g-h\|<1$. Therefore, $|\Lambda_k|\leq \dens(\Z_k)$, and we are done.
\end{proof}

Write $\G=\{g_\alpha\}_{\alpha<\Gamma}$. We now build by transfinite induction a continuous increasing chain $(\Z_\alpha)_{\alpha< \Gamma}$ of subspaces of $\X$ such that $\dens(\Z_\alpha)\leq \max\{|\alpha|,\omega\}$, $g_\alpha\in \Z_{\alpha+1}$, and each $\Z_\alpha$ satisfies \eqref{eq: dagger}. In fact, at a successor level we apply \Cref{claim: closing-off} to $\spn\{\Z_\alpha, g_\alpha \}$ to obtain the subspace $\Z_{\alpha+1}$; if $\gamma< \Gamma$ is a limit ordinal, we set $\Z_\gamma\coloneqq \overline{\bigcup_{\alpha< \gamma}\Z_\alpha}$. The fact that each $\Z_\alpha$ satisfies \eqref{eq: dagger} easily implies that $\Z_\gamma$ satisfies it too, hence the induction is complete.

To conclude, set $\G_\alpha\coloneqq \G\cap \Z_\alpha$, so that $\G= \bigcup_{\alpha<\Gamma} \G_\alpha$. Since $\dens(\Z_\alpha)<\Gamma$ and $\G_\alpha\subseteq \Z_\alpha$ is discrete, $|\G_\alpha|<\Gamma$; therefore, $\G_\alpha$ is free by assumption. Further, \eqref{eq: dagger} implies that $\G_{\alpha+1}/\G_\alpha$ is a discrete subgroup of $\Z_{\alpha+1}/\Z_\alpha$. Thus, our assumption also implies that $\G_{\alpha+1}/ \G_\alpha$ is free. Finally, the chain $(\G_\alpha)_{\alpha< \Gamma}$ is continuous, because for a limit ordinal $\gamma$ we have
\[ \bigcup_{\alpha<\gamma} \G_\alpha= \G\cap \bigcup_{\alpha<\gamma} \Z_\alpha= \G\cap \overline{\bigcup_{\alpha<\gamma} \Z_\alpha}= \G\cap \Z_\gamma. \]
Here, the intermediate equality follows from the validity of \eqref{eq: dagger} for each $\Z_\alpha$ and the last one from the continuity of the chain $(\Z_\alpha)_{\alpha< \Gamma}$. Therefore, we are in position to apply \Cref{fact: free by chain}, which yields that $\G$ is free and concludes the proof.
\end{proof}

\section{Open Problems}\label{sec: problems}
In this last section we formulate or reiterate some open problems related to our results from the previous sections. We begin with some questions that stem from \Cref{mainth: ell2 body} and concern Hilbert spaces.

In the light of \Cref{probl: FL} and \Cref{mainth: ell2 body}, it is natural to ask whether a tiling by balls might possibly exist also in a separable Hilbert space. Notice that Klee's approach from \cite{Klee1} and ours from \Cref{mainth: ell2 body} are inherently non-separable, because no countable proximinal subset of $\ell_2$ is $r$-dense, \cite{FonfLind}*{Proposition 2.1}. A positive answer would provide a substantial improvement of Preiss' normal tiling of $\ell_2$ constructed in \cite{Preiss}.

\begin{problem} Does there exist an equivalent norm $\nn\cdot$ on $\ell_2$, such that $(\ell_2,\nn\cdot)$ admits a (lattice) tiling by balls of unit radius? More generally, can $\ell_2$ be tiled by translates of a bounded convex body?
\end{problem}

Klee's tiling \cites{Klee1} in $\ell_1(\Gamma)$ that inspired our construction had the additional feature to be disjoint. While we saw that our method cannot produce disjoint tilings (\Cref{prop: lattice not disjoint}), it remains open whether a disjoint tiling of a Hilbert space might be possible at all (notice that a disjoint tiling is only possible for normed spaces of density at least $\mathfrak{c}$, \cite{Klee1}*{Proposition 3.3.}). Hence, we repeat (and rephrase) the following problem from \cite{Klee1}*{Question 3.4}.

\begin{problem}\label{probl: disjointtiling} Does there exists a disjoint tiling by bounded convex bodies in $\ell_2(\mathfrak{c})$? 
\end{problem}

We now ask a couple of problems relative to \Cref{prop: point-countable}. Recall that $\ell_2(\Gamma)$ does not admit a point-finite covering by balls whenever $|\Gamma|<\mathfrak{c}$, \cites{FZHilbert, DEBEpointfinite}. On the other hand, in \Cref{prop: point-countable}, we proved that $\ell_2(\mathfrak{c})$ admits a point-countable covering with balls (of radius $1$). Therefore, the long-standing question (see, \emph{e.g.}, \cite{Klee1}*{Question 2.6}, or \cite{DEBEpointfinite}*{Problem 3.4}) whether $\ell_p(\mathfrak{c})$ ($1<p<\infty$) admits a point-finite covering by balls remains open.

We also take this opportunity to point out that in \cite{Klee2}*{Theorem 3.2} (and subsequently repeated in \cite{Zanco_carpet}) it is claimed that Klee's tiling of $\ell_p(\Gamma)$ from \cite{Klee1} is point-finite. The argument in \cite{Klee2} is however erroneous, since it uses, misquoting \cite{BRR}, the incorrect fact that the unit ball of $\ell_p(\Gamma)$ does not contain a $(2^{1/p}+)$-separated sequence.

\begin{problem}\label{probl: point-finite} Does there exists a point-finite covering by closed balls in $\ell_2(\mathfrak{c})$?
\end{problem}

While in \Cref{prop: point-countable} we obtained both the tiling and the covering claims at once, the existence of point-finite coverings by balls or point-finite tilings by balls of an equivalent norm are in principle different issues. To the best of our knowledge, the following is open for any infinite set $\Gamma$.

\begin{problem} Is there a norm $\nn\cdot$ on $\ell_2(\Gamma)$ such that $(\ell_2(\Gamma), \nn\cdot)$ admits a point-finite tiling by balls? More generally, is there an infinite-dimensional reflexive Banach space that admits a point-finite tiling (respectively, covering) with balls?
\end{problem}

We finish with one problem for a general normed space. As we said in the Introduction, separable Banach spaces do not admit tilings by rotund or smooth bounded convex bodies, \cites{KleeTri, KleeMalZan}. For non-separable spaces, the same is only known for LUR or Fr\'echet smooth Banach spaces, \cite{DEVEtiling}. As we saw in \Cref{thm: not rotund &co}, our approach can't be employed to obtain rotund or smooth bodies, therefore the following remains open.

\begin{problem}\label{probl: rotund/smooth} Does there exist a rotund and/or smooth normed space $\X$ that admits a tiling by balls?
\end{problem}

\noindent\textbf{Acknowledgements.} We are most grateful to Esteban Mart\'inez Va\~{n}\'o and Abraham Rueda Zoca for suggesting us the validity of \Cref{thm: tiling if ell1 inside} and for pointing out to us the relevance of \cite{AMR_ell1(kappa)}. The second-named author is also indebted with Antonio Avil\'es and Grzegorz Plebanek for discussing some aspects of the argument in \cite{KaniaKostana}. Such a conversation gave us the key idea for the alternative proof of \Cref{thm: discrete is free} and we are also grateful to the Erwin Schr\"{o}dinger Institute in Vienna, where this conversation took place, for the excellent working conditions. Further, we thank Piero Papini and Clemente Zanco for providing us with a copy of the paper \cite{Rog84} and Michal Doucha for pointing out to us the reference \cite{DOSZ}.


\end{document}